\documentclass[a4paper,reqno]{amsart}

\usepackage{amsmath}
\usepackage{amssymb}
\usepackage{amscd}
\usepackage{amsthm}
\usepackage{needspace}
\usepackage{type1cm}
\usepackage{tcolorbox}
\usepackage{mathrsfs}

\usepackage[colorlinks,citecolor=darkgreen,linkcolor=red]{hyperref}
\definecolor{darkgreen}{rgb}{0.0, 0.6, 0.0}

\renewcommand\baselinestretch{1.1}

\usepackage[all]{xy}
\input xy
\xyoption{all}
\tcbuselibrary{breakable, skins, theorems}

\def\A{\mathcal{A}}
\def\C{\mathcal{C}}
\def\D{\mathcal{D}}
\def\E{\mathcal{E}}
\def\F{\mathcal{F}}
\def\G{\mathcal{G}}
\def\H{\mathcal{H}}
\def\L{\mathcal{L}}

\def\O{\mathcal{O}}
\def\T{\mathcal{T}}

\def\V{\mathcal{V}}

\DeclareMathOperator{\md}{\mathsf{mod}}
\renewcommand{\mod}{\md}

\DeclareMathOperator{\add}{\mathsf{add}}

\DeclareMathOperator{\Filt}{\mathsf{Filt}}
\DeclareMathOperator{\thick}{\mathsf{thick}}
\DeclareMathOperator{\Coh}{\mathsf{Coh}}
\DeclareMathOperator{\vect}{\mathsf{vect}}

\DeclareMathOperator{\Hom}{Hom}
\DeclareMathOperator{\sheafHom}{\H{\it om}}
\DeclareMathOperator{\sheafEnd}{\E{\it nd}}
\DeclareMathOperator{\RHom}{\mathbb{R}Hom}
\DeclareMathOperator{\End}{End}
\DeclareMathOperator{\Ext}{Ext}
\DeclareMathOperator{\per}{per}

\DeclareMathOperator{\Ker}{Ker}

\DeclareMathOperator{\Div}{Div}
\DeclareMathOperator{\Spec}{Spec}
\DeclareMathOperator{\Sing}{Sing}
\DeclareMathOperator{\RG}{\mathbb{R}\Gamma}
\DeclareMathOperator{\ch}{char}
\DeclareMathOperator{\Tot}{Tot}
\DeclareMathOperator{\tr}{tr}
\DeclareMathOperator{\rk}{rk}

\def\gl{\mathop{\rm gl.dim}\nolimits}

\theoremstyle{definition}
\newtheorem{Thm}{Theorem}[section]
\newtheorem{Lem}[Thm]{Lemma}
\newtheorem{Prop}[Thm]{Proposition}
\newtheorem{Cor}[Thm]{Corollary}
\newtheorem{Def}[Thm]{Definition}
\newtheorem{Ex}[Thm]{Example}
\newtheorem{Rem}[Thm]{Remark}
\newtheorem{Conj}[Thm]{Conjecture}

\newcommand{\FRAC}[2]{\leavevmode\kern.1em\raise.5ex\hbox{\the\scriptfont0 #1}\kern-.1em/\kern-.15em\lower.25ex\hbox{\the\scriptfont0 #2}}

\title{Weak del Pezzo surfaces are characterized by the existence of $2$-tilting bundles}
\author{Ryu Tomonaga}
\address{Graduate School of Mathematical Sciences, The University of Tokyo, 3-8-1 Komaba, Meguro-ku, Tokyo, 153-8914, Japan}
\email{ryu-tomonaga@g.ecc.u-tokyo.ac.jp}

\begin{document}
\begin{abstract}
Tilting bundles provide a fundamental bridge between algebraic geometry and representation theory. For a tilting bundle on a regular proper $d$-dimensional variety over any field, the global dimension of its endomorphism algebra is at least $d$, and a case of particular interest occurs when this lower bound is attained. Such a tilting bundle, called a $d$-tilting bundle, fits into the framework of the derived McKay correspondence and higher Auslander--Reiten theory.

The first main result of this paper shows that the existence of such a bundle forces the variety to be weak Fano: more precisely, if a regular proper $d$-dimensional variety over any field admits a $d$-tilting bundle, then its anti-canonical bundle is semiample and big. As a consequence, the endomorphism algebra of a $d$-tilting bundle is $d$-representation tame, so the geometry naturally produces higher-dimensional analogues of extended Dynkin quivers.

Second, we prove a converse in dimension two: every weak del Pezzo surface over an algebraically closed field admits a $2$-tilting bundle. Together, these results prove the variety case of Chan's conjecture \cite{Cha}, originally formulated in characteristic zero, and extend it to arbitrary characteristic: a smooth projective surface admits a $2$-tilting bundle if and only if it is a weak del Pezzo surface.

As an application, we show that the anti-canonical ring of any regular proper $d$-dimensional variety over any field admitting a $d$-tilting bundle has a non-commutative crepant resolution (NCCR), given by the endomorphism algebra of the pulled-back bundle on the total space of the canonical bundle. In particular, over algebraically closed fields we construct NCCRs of anti-canonical cones over Du Val del Pezzo surfaces. These are obtained as the $3$-Calabi--Yau completions of the endomorphism algebras of $2$-tilting bundles on the corresponding weak del Pezzo surfaces, extending the known construction for smooth del Pezzo surfaces to the Du Val case.
\end{abstract}

\maketitle
\tableofcontents

\section*{Introduction}

For simplicity, throughout this Introduction, $k$ is an algebraically
closed field, with no restriction on its characteristic. We state the
results in this setting; the general results in
Section~\ref{sec:higher-representation} and
Subsections~\ref{subsec:necessity} and \ref{subsec:anticanonical-nccr}
are proved over arbitrary fields.

Tilting theory provides a mechanism for translating geometry into
algebra: a tilting bundle $\E$ on a smooth proper variety $X$ over $k$
yields a derived equivalence
\[
    \D^b(\Coh X)\simeq\D^b(\mod\End_X(\E)).
\]
Under this equivalence, geometric information about $X$ can be studied
through representation theory of the finite-dimensional algebra
$\End_X(\E)$, and representation-theoretic structures can be realized
geometrically. Beilinson first constructed tilting bundles on projective
space $\mathbb P^d$ \cite{Bei}. Many examples have since been constructed,
including in the study of the McKay correspondence and non-commutative
crepant resolutions (NCCRs) \cite{BH,HIMO,HP14,Kap,Tom25b}.

We illustrate these connections with projective space. The endomorphism
algebra of Beilinson's tilting bundle
$\E:=\bigoplus_{i=0}^d\O_{\mathbb P^d}(i)$ has global dimension $d$.
For $d=1$, it is the path algebra
$k[\xymatrix@C=15pt{\circ \ar@2[r] & \circ}]$ of the Kronecker quiver.
Let $\pi:T:=\Tot(\omega_{\mathbb P^d})\to\mathbb P^d$ be the total
space of the canonical bundle. The pullback $\pi^*\E$ is tilting on $T$,
and its endomorphism algebra is the $(d+1)$-preprojective algebra
$\Pi:=\Pi_{d+1}(\End_{\mathbb P^d}(\E))$. The anti-canonical ring is
\[
    R:=\Gamma(T,\O_T)
      =\bigoplus_{n\geq0}H^0(\mathbb P^d,\omega_{\mathbb P^d}^{-n})
      \cong k[x_0,\ldots,x_d]^{(d+1)}.
\]
The morphism $T\to\Spec R$ is a crepant resolution and $\Pi$ is an
NCCR of $R$, giving a derived McKay-type equivalence. When $d=1$ and
$\ch k=0$, this recovers the classical $A_1$ example of the derived
McKay correspondence \cite{BKR}. Keller's Calabi--Yau completion
\cite{Kel11} gives a theoretical explanation for the general construction.

We seek a generalization of these phenomena. If a smooth proper
$d$-dimensional variety $X$ has a tilting bundle $\E$, then
$d\leq\gl\End_X(\E)\leq2d$ (Proposition \ref{SerreRI}; see also
\cite{BF}). Buchweitz and Hille showed that
$\gl\End_X(\E)=d$ holds if and only if the pullback to
$\Tot(\omega_X)$ remains tilting \cite{BuH}; we give a complete
proof in Proposition \ref{dtiltsheaf}.
Such a bundle is called a \emph{$d$-tilting bundle} \cite{HIMO}.

\begin{Thm}[Proposition \ref{dtiltsheaf}]
Let $X$ be a smooth proper $d$-dimensional variety over $k$, with $d\geq1$,
and let $\E$ be a tilting bundle. The following conditions are equivalent.
\begin{enumerate}
\item $\E$ is $d$-tilting, i.e.\ $\gl\End_X(\E)=d$.
\item For the projection $\pi:\Tot(\omega_X)\to X$, the pullback
$\pi^*\E$ is tilting on $\Tot(\omega_X)$.
\end{enumerate}
Under these conditions, $\End_X(\E)$ is $d$-representation infinite.
\end{Thm}

A \emph{$d$-representation infinite algebra} is a finite-dimensional
algebra of global dimension $d$ introduced in \cite{HIO}, generalizing
non-Dynkin path algebras through higher Auslander--Reiten theory.
Its $(d+1)$-preprojective algebra is a tensor algebra and, over our
algebraically closed field, is the $(d+1)$-Calabi--Yau completion
concentrated in cohomological degree zero \cite{HIO,Kel11}.

A central problem is to determine which varieties admit $d$-tilting
bundles. A smooth projective curve over $k$ admits a $1$-tilting bundle
if and only if it is $\mathbb P^1$. For surfaces, Chan formulated a
conjecture over an algebraically closed field of characteristic zero,
in the more general setting of weighted projective surfaces
\cite[Conjecture~1.1]{Cha}. We consider its variety case over algebraically closed fields
of arbitrary characteristic.

\begin{Conj}[Variety form of Chan's conjecture]\label{introconj}
For a smooth projective surface $X$ over $k$, the following are equivalent.
\begin{enumerate}
\item $X$ has a $2$-tilting bundle.
\item $X$ is a weak del Pezzo surface, i.e.\ $-K_X$ is nef and big.
\end{enumerate}
\end{Conj}

Hille and Perling constructed tilting bundles on all smooth projective
rational surfaces \cite[Theorem~1.1]{HP14}. The $2$-tilting condition
is stronger and is closely tied to positivity of $-K_X$.
Every del Pezzo surface admits a $2$-tilting bundle
\cite[Proposition~7.3 and Remark~7.4]{VdB04a}, as do many weak del Pezzo
surfaces \cite{BHI,EXZ,HP11}. For the necessary direction, Chan proved
that if a smooth projective surface over an algebraically closed field
of characteristic zero admits a $2$-tilting bundle, then $-K_X$ is nef
\cite[Theorem~1.2]{Cha}.

We prove Conjecture \ref{introconj} in arbitrary characteristic.
The necessary direction follows from a result in every dimension.

\begin{Thm}[Theorem \ref{nece}, Corollary \ref{RPifin}]\label{intronece}
Let $X$ be a smooth proper $d$-dimensional variety over $k$, with $d\geq1$,
and let $\E$ be a $d$-tilting bundle.
\begin{enumerate}
\item $\omega_X^{-1}$ is semiample and big. In particular, $X$ is weak Fano.
\item $H^i(X,\omega_X^{-n})=0$ for every $i>0$ and $n\geq0$.
\item $\End_X(\E)$ is $d$-representation tame.
\end{enumerate}
\end{Thm}

The class of \emph{$d$-representation tame algebras} forms a subclass
of $d$-representation infinite algebras and generalizes path algebras
of extended Dynkin quivers. Theorem \ref{intronece} also gives a partial answer to the
conjecture that a smooth projective variety admitting a tilting bundle
is rational: a smooth projective surface admitting a $2$-tilting bundle
is rational (Corollary \ref{rat}).

For the converse in dimension two, we prove the following.

\begin{Thm}[Theorem \ref{suff}]\label{introsuff}
Every weak del Pezzo surface $X$ over $k$ admits a $2$-tilting bundle
$\E$ containing $\O_X$ as a direct summand.
\end{Thm}

Combining Theorems \ref{intronece} and \ref{introsuff} gives the following.

\begin{Cor}
Conjecture \ref{introconj} is true.
\end{Cor}

Theorem \ref{introsuff} also has an application to NCCRs. Introduced by
Van den Bergh \cite{VdB04a}, NCCRs are an algebraic counterpart of crepant
resolutions and extend the setting of the derived McKay correspondence
\cite{BKR}. They have been constructed for quotient singularities
\cite{Iya07a,Tom24}, compound Du Val singularities admitting crepant
resolutions \cite{VdB04b}, and various toric and other singularities
\cite{Bro,HN,Mat22,Tom25d,SVdB20a,SVdB20b,Han25,Har,SVdB17}.
Their mutations are studied in \cite{HH24,IW,NVdB26}.

We first extend the Fano construction in
\cite[Proposition~7.2]{VdB04a} as follows.

\begin{Thm}[Theorem \ref{NCCRanticanonical}]
Let $X$ be a smooth proper $d$-dimensional variety over $k$, with $d\geq1$,
and let $\E$ be a $d$-tilting bundle. Put
$\pi:T:=\Tot_X(\omega_X)\to X$ and
$R:=\bigoplus_{n\geq0}H^0(X,\omega_X^{-n})$. Then $R$ is a
$(d+1)$-dimensional Gorenstein normal domain and
\[
    \End_T(\pi^*\E)
    \cong\End_R\bigl(\Gamma(T,\pi^*\E)^{**}\bigr)
\]
is an NCCR of $R$.
\end{Thm}

Van den Bergh constructed NCCRs of cones over smooth del Pezzo surfaces
in \cite[Proposition~7.3 and Remark~7.4]{VdB04a}. These NCCRs have
recently been studied extensively in \cite{NVdB26}. Applying the general
construction to a minimal resolution and using Theorem \ref{introsuff}
gives an extension to Du Val del Pezzo cones.

\begin{Thm}[Theorem \ref{NCCRsingdP}]
Let $X$ be a Du Val del Pezzo surface over $k$. Then its anti-canonical
ring $R:=\bigoplus_{n\geq0}H^0(X,\omega_X^{-n})$ admits an NCCR.
\end{Thm}

This NCCR is the $3$-Calabi--Yau completion of the endomorphism algebra
of a $2$-tilting bundle on the corresponding weak del Pezzo surface.
That finite-dimensional algebra is $2$-representation tame.
The construction gives a derived McKay-type equivalence, realized by
the tilting bundle on the total space of the canonical bundle.

In Appendix~\ref{app:rigid-sheaves}, we extend Kuleshov's structure
theorem for rigid torsion-free sheaves on weak del Pezzo surfaces to
arbitrary characteristic. Every indecomposable rigid torsion-free sheaf
has a filtration by exceptional bundles of the same slope
(Theorem \ref{rigtorfss}). In positive characteristic, a smooth
anti-canonical curve need not exist, so the proof uses an integral
anti-canonical curve instead.

\subsection*{Conventions}
In the main body, unless explicitly stated otherwise, $k$ is an arbitrary field. All varieties, categories and algebras are defined over $k$, and categories and functors are $k$-linear. A variety is an integral separated scheme of finite type over $k$; geometric integrality is not assumed. Section~\ref{sec:higher-representation} and Subsections~\ref{subsec:necessity} and \ref{subsec:anticanonical-nccr} work over arbitrary fields. In Section~\ref{sec:preliminaries}, Subsections~\ref{subsec:sufficiency} and \ref{subsec:du-val-cones}, and Appendix~\ref{app:rigid-sheaves}, $k$ is assumed to be algebraically closed, with no restriction on its characteristic.

For a $d$-dimensional variety $X$, $\omega_X$ denotes its dualizing sheaf over $k$. If $X$ is Gorenstein, then $\omega_X$ is a line bundle and its dualizing complex over $k$ is $\omega_X[d]$. If $X$ is smooth over $k$, this is the usual canonical bundle $\bigwedge^d\Omega^1_{X/k}$. We call a regular proper variety weak Fano if $\omega_X^{-1}$ is nef and big. Over an algebraically closed field, a weak del Pezzo surface is a smooth projective surface with nef and big anti-canonical bundle. We write $\vect X$ for the category of vector bundles on $X$, and tilting sheaves on a regular variety are understood as tilting objects of $\D^b(\Coh X)=\per X$.

We write $D:=\Hom_k(-,k)$ and use right modules unless stated otherwise. For a finite-dimensional algebra $A$ of finite global dimension, let $\nu:=-\otimes_A^{\mathbb L}DA$ denote the Serre functor of $\per A=\D^b(\mod A)$. For an integer $d\geq0$, put $\nu_d:=\nu\circ[-d]$.

\subsection*{Use of AI}
During the preparation of this work, the author used ChatGPT-5.5 Pro and
ChatGPT-5.5 Thinking as auxiliary tools to
discuss possible approaches and alternative formulations for specific
arguments, namely Propositions \ref{prfsemiample}, \ref{antcanvan}
and \ref{prfbig}, and to improve the
exposition of preliminary drafts. During revision, Codex was also used
for literature checks and editing. The final proofs were independently verified and finalized
by the author. The author takes full responsibility for the accuracy,
originality and integrity of the paper.

\section*{Acknowledgements}
The author expresses his gratitude to Wahei Hara, Akira Ishii and Hirotaka Onuki for helpful answers to his questions. He is also grateful to Norihiro Hanihara, Lutz Hille, Osamu Iyama, Tatsuro Kawakami, Yusuke Nakajima, Shu Nimura and Kenta Sato for fruitful discussions. This work was supported by the WINGS-FMSP program at the Graduate School of Mathematical Sciences, the University of Tokyo, and JSPS KAKENHI Grant Number JP25KJ0818.

\section{Higher representation infinite algebras in projective geometry}\label{sec:higher-representation}

Throughout this section, $k$ is an arbitrary field. We recall higher representation infinite algebras introduced by \cite{HIO} and their relation to regular proper varieties. The construction uses Serre duality and does not require smoothness over $k$.

\begin{Def}\cite[Definition~2.7]{HIO}
Let $A$ be a finite-dimensional algebra. For $d\geq1$, $A$ is called {\it $d$-representation infinite} if $\gl A\leq d$ and 
\[\nu_d^{-n}(A)\in\mod A\subseteq\per A\]
holds for all $n\geq0$.
\end{Def}

For a finite-dimensional algebra $A$ with $\gl A\leq d$, its
\emph{$(d+1)$-preprojective algebra} is the graded tensor algebra
\[
    \Pi_{d+1}(A):=T_A\Ext_A^d(DA,A).
\]
If $A$ is $d$-representation infinite, its homogeneous component of
degree $n$ is naturally $\nu_d^{-n}(A)$ as a right $A$-module
\cite[Section~4.2]{HIO}. In particular,
\[
    \Pi_{d+1}(A)\cong
    \bigoplus_{n\geq0}\Hom_{\per A}(A,\nu_d^{-n}A),
\]
with multiplication given by composition and the action of
$\nu_d^{-1}$.

\begin{Rem}\label{CYseparable}
If $A/\operatorname{rad}A$ is separable over $k$, then finite global
dimension implies that $A$ is homologically smooth over $k$.
In this case, $d$-representation infiniteness is equivalent to the
$(d+1)$-Calabi--Yau completion being concentrated in cohomological
degree $0$, and that completion is $\Pi_{d+1}(A)$
\cite{Kel11,HIO}. This hypothesis holds over every perfect field.
Over arbitrary fields we use the tensor algebra above, without
asserting bimodule Calabi--Yau smoothness.
\end{Rem}

For basic properties and importance of higher representation infinite algebras, see \cite{HIO}. To see connections between these algebras and projective geometry, we introduce the notion of $d$-tilting object.

\begin{Def}
Let $\T$ be a triangulated category. Take an object $X\in\T$.
\begin{enumerate}
\item $X$ is called {\it pretilting} if $\T(X,X[\neq0])=0$ holds.
\item $X$ is called {\it tilting} if it is pretilting and $\thick X=\T$ holds.
\item For $d\geq1$, $X$ is called {\it d-tilting} if it is tilting and $\gl\End_{\T}(X)\leq d$ holds.
\end{enumerate}
\end{Def}

As in \cite[Proposition~7.14]{HIMO} and \cite{BuH}, we can prove the following. For (1), see also \cite{BF}.

\begin{Prop}\label{SerreRI}
Let $\A$ be a nonzero Hom-finite noetherian abelian category and $T\in\A$ a tilting object of $\D^b(\A)$. Let $d\geq1$. Assume $\A$ has an exact auto-equivalence $F\colon\A\xrightarrow[\simeq]{}\A$ such that $F[d]\colon\D^b(\A)\xrightarrow[\simeq]{}\D^b(\A)$ is a Serre functor.
\begin{enumerate}
\item We have $d\le\gl\End_\A(T)\le 2d$.
\item\cite{BuH} If $\gl\End_\A(T)=d$, then $\End_\A(T)$ is $d$-representation infinite.
\end{enumerate}
\end{Prop}

For the convenience of the reader, we give a complete proof here. To prove this proposition, the following lemma is useful.

\begin{Lem}\label{gldimd}
Let $A$ be a finite-dimensional algebra of finite global dimension. For $d\geq0$, the following conditions are equivalent.
\begin{enumerate}
\item $\gl A\leq d$
\item $H^{>0}(\nu_d^{-1}(A))=0$
\item $\nu_d^{-1}(\D^{b,\leq0}(A))\subseteq\D^{b,\leq0}(A)$
\end{enumerate}
\end{Lem}
\begin{proof}
(1)$\Leftrightarrow$(2) follows from $H^i(\nu_d^{-1}A)\cong\Ext_A^{d+i}(DA,A)$. (3)$\Rightarrow$(2) is obvious. For (1)$\Rightarrow$(3), see \cite[Proposition~5.4(a)]{Iya11}.
\end{proof}

\begin{proof}[Proof of Proposition \ref{SerreRI}]
Serre duality gives
\[
    \Ext_\A^i(U,V)\cong D\Ext_\A^{d-i}(V,FU)=0
    \qquad(i>d,\ U,V\in\A).
\]
Thus $\A$ has finite global dimension. Since $\A$ is noetherian,
the endomorphism algebra $A:=\End_\A(T)$ also has finite global
dimension by \cite[Theorem~1]{KK20}. We may therefore apply
Lemma \ref{gldimd} below.
Fix a triangle equivalence $\D^b(\A)\simeq\per A$. By uniqueness of the Serre functor, the following diagram commutes up to a natural isomorphism.
\[\xymatrix{
\D^b(\A) \ar[r]_\simeq \ar[d]_{F[d]} & \per A \ar[d]^\nu\\
\D^b(\A) \ar[r]_\simeq & \per A
}\]

(1) By the above commutative diagram, we have
\[H^{>0}(\nu_{2d}^{-1}A)\cong\Ext_\A^{>0}(T,F^{-1}T[d])\cong D\Ext_\A^{<0}(F^{-2}T,T)=0.\]
Thus by Lemma \ref{gldimd} (2)$\Rightarrow$(1), we have $\gl A\le 2d$. Suppose $\gl A<d$ holds. Then by Lemma \ref{gldimd} (1)$\Rightarrow$(2), we have
\[H^{>0}(\nu_{d-1}^{-1}(A))=0.\]
Since $F^{-1}T\in\A$, we have
\[H^{\le0}(\nu_{d-1}^{-1}(A))\cong\Ext_\A^{<0}(T,F^{-1}T)=0.\]
These imply $\nu_{d-1}^{-1}(A)=0$, which is a contradiction. Thus we obtain $\gl A\ge d$.

(2) Take $n\geq0$. Since $F^{-n}T\in\A$, we have
\[H^{<0}(\nu_d^{-n}(A))\cong\Ext_\A^{<0}(T,F^{-n}T)=0.\]
Since $\gl A\leq d$, by Lemma \ref{gldimd} (1)$\Rightarrow$(3), we have
\[H^{>0}(\nu_d^{-n}(A))=0.\qedhere\]
\end{proof}

With these preparations, we can state connections between higher representation infinite algebras and projective geometry in the following way.

\begin{Prop}[cf.\ \cite{BuH}]\label{dtiltsheaf}
Let $X$ be a $d$-dimensional regular proper variety over $k$, with $d\geq1$. For a tilting sheaf $\F\in\Coh X$, the following conditions are equivalent.
\begin{enumerate}
\item $\F$ is a $d$-tilting sheaf.
\item $\Ext_X^{>0}(\F, \F\otimes_X\omega_X^{-1})=0$ holds.
\item $\Ext_X^{>0}(\F, \F\otimes_X\omega_X^{-n})=0$ holds for all $n\geq0$.
\item Let $\pi\colon \Tot(\omega_X)\to X$ be the total space of the canonical bundle. Then $\pi^*\F\in\Coh \Tot(\omega_X)$ is a tilting sheaf.
\end{enumerate}
If this is the case, then $\End_X(\F)$ is a $d$-representation infinite algebra. Moreover, we have
\[\End_{\Tot(\omega_X)}(\pi^*\F)\cong\Pi_{d+1}(\End_X(\F)).\]
\end{Prop}
\begin{proof}
The Serre functor on $\D^b(\Coh X)$ is
$-\otimes\omega_X[d]$. Proposition \ref{SerreRI} ensures finite
global dimension of $A:=\End_X(\F)$, and Lemma \ref{gldimd}
gives (1)$\Leftrightarrow$(2)$\Leftrightarrow$(3).
For (3)$\Leftrightarrow$(4), put $T:=\Tot_X(\omega_X)$.
Since $X$ and $T$ are regular, $\F$ and $\pi^*\F$ are perfect.
The affine adjunction for $\pi$ shows that $\pi^*\F$ generates
the derived category of quasi-coherent sheaves on $T$, and hence
its perfect subcategory $\D^b(\Coh T)$. Moreover,
\[
    \Ext_T^i(\pi^*\F,\pi^*\F)
    \cong\bigoplus_{n\geq0}
    \Ext_X^i(\F,\F\otimes\omega_X^{-n}),
\]
so the required vanishing is precisely (3); negative Ext groups
vanish because $\pi^*\F$ is a sheaf. This also explains why the
argument of \cite[Lemma~3.19]{BF} works in the regular case over any
field. Under these conditions, $A$ is $d$-representation infinite
by Proposition \ref{SerreRI}. The graded algebra isomorphism follows from
\[\End_{\Tot(\omega_X)}(\pi^*\F)\cong\bigoplus_{n\geq0}\Hom_X(\F,\F\otimes\omega_X^{-n})\cong\bigoplus_{n\geq0}\Hom_A(A,\nu_d^{-n}A)\cong\Pi_{d+1}(A).\qedhere\]
\end{proof}

\begin{Rem}
A tilting sheaf satisfying (2),(3) and (4) in Proposition \ref{dtiltsheaf} is called almost pullback \cite{Zha}, $d$-hereditary \cite{Cha} and pullback \cite{BF} respectively. This proposition says that these conditions are all equivalent to being a $d$-tilting sheaf. See also \cite[Theorem~1.9]{Zha}.
\end{Rem}

\begin{Rem}\label{FanoNCCR}
Put $R:=\bigoplus_{n=0}^\infty\Gamma(X,\omega_X^{-n})$. Van den Bergh shows that if $X$ is Fano and $\E\in\vect X$ is a $d$-tilting bundle, then $\End_{\Tot(\omega_X)}(\pi^*\E)$ gives a non-commutative crepant resolution (NCCR) of $R$ \cite[Proposition~7.2]{VdB04a}. Theorem \ref{NCCRanticanonical} extends this construction to arbitrary regular proper varieties admitting $d$-tilting bundles.
\end{Rem}

A typical example is given by Beilinson's tilting bundles.

\begin{Ex}
Put $X:=\mathbb{P}^d$. Then $\F:=\bigoplus_{i=0}^d\O_X(i)$ is a tilting bundle by \cite{Bei}. Here, we can easily check $\Ext_X^{>0}(\F,\F(d+1))=0$. Therefore the $d$-Beilinson algebra $\End_X(\F)$ is $d$-representation infinite. See also \cite[Example~2.15]{HIO}.
\end{Ex}

Finally, we recall the definition of higher representation tame algebras.

\begin{Def}\cite[Definition~6.10]{HIO}
A $d$-representation infinite algebra $A$ is called {\it $d$-representation tame} if its $(d+1)$-preprojective algebra $\Pi_{d+1}(A)$ is module-finite over a commutative noetherian ring acting centrally. Following \cite{HIO}, this is the meaning of a \emph{noetherian algebra} in this definition.
\end{Def}

This generalizes the path algebras of extended Dynkin quivers. By Remark \ref{FanoNCCR}, if a $d$-dimensional smooth Fano variety has a $d$-tilting bundle, then its endomorphism algebra is $d$-representation tame. Corollary \ref{RPifin} proves the same assertion for $d$-tilting bundles on regular proper $d$-dimensional varieties over every field.

\section{Preliminaries}\label{sec:preliminaries}

Throughout this section, $k$ is algebraically closed of arbitrary
characteristic. These preliminaries concern the construction on
weak del Pezzo surfaces in Subsection~\ref{subsec:sufficiency}.

\subsection{Exceptional sequences and universal (co)extensions}\label{subsec:exceptional-sequences}

Recall that an object $E$ of a $k$-linear triangulated category $\T$
is exceptional if $\End_{\T}(E)=k$ and
$\T(E,E[p])=0$ for every $p\neq0$. An ordered collection
$(E_1,\ldots,E_n)$ is exceptional if its terms are exceptional and
$\T(E_j,E_i[p])=0$ for every $j>i$ and $p\in\mathbb Z$.
It is full if $\thick\{E_1,\ldots,E_n\}=\T$.

We also fix the extension notation used below. For full subcategories
$\mathcal U,\mathcal V\subseteq\T$, write $\mathcal U*\mathcal V$
for the full subcategory of objects $M$ occurring in a triangle
$U\to M\to V\to U[1]$ with $U\in\mathcal U$ and
$V\in\mathcal V$. The octahedral axiom makes this operation
associative. For a collection $\mathcal S$ of objects, let
$\Filt_{\T}\mathcal S$ be its closure under finite iterated
extensions, including the zero object; direct summands are not
implicitly added. In an abelian category $\mathcal A$, we instead
write $\Filt_{\mathcal A}\mathcal S$ for objects with a finite
filtration by subobjects whose successive quotients belong to
$\mathcal S$. When $\mathcal A$ is the heart of a bounded
$t$-structure on $\T$, these two notions agree on $\mathcal S\subseteq
\mathcal A$, since the relevant triangles are short exact sequences.
For sheaves on $X$, we abbreviate
$\Filt_{\Coh X}\mathcal S$ to $\Filt\mathcal S$.

For line bundles, the following is the augmentation at the first
term of \cite[Proposition~2.4]{HP14}. Its formulation for derived
objects follows from Orlov's blowup formula
\cite[Theorem~4.3]{Orl}; here $\mathbb L\pi^*$ denotes derived pullback.

\begin{Prop}\label{excep}
Let $X$ be a smooth projective surface and $\pi\colon Y\to X$ the blowing up at a closed point $P\in X$. Put $E:=\pi^{-1}(P)\subseteq Y$. Assume that we have an exceptional collection $(\E_1,\cdots ,\E_n)$ on $\D^b(\Coh X)$ with $\E_1=\O_X$.
\begin{enumerate}
\item $(\mathbb L\pi^*\E_1=\O_Y,\O_Y(E),\mathbb L\pi^*\E_2,\cdots ,\mathbb L\pi^*\E_n)$ is an exceptional collection on $\D^b(\Coh Y)$.
\item If $(\E_1,\cdots ,\E_n)$ is full, then so is $(\mathbb L\pi^*\E_1,\O_Y(E),\mathbb L\pi^*\E_2,\cdots ,\mathbb L\pi^*\E_n)$.
\end{enumerate}
\end{Prop}

The next result is the triangulated version of the universal extension
construction of \cite[Theorems~4.3 and 4.4]{HP14}. We give the
construction and proof, so that no abelian-category hypothesis is
needed.

\begin{Prop}\label{uniext}\cite[Theorem~4.4]{HP14}
Let $\T$ be a Hom-finite triangulated category with a full exceptional
collection $(E_1,\ldots,E_n)$. Assume that
$\T(E_i,E_j[p])=0$ for all $i,j$ and $p\notin\{0,1\}$.
Then there are objects
$E_i'\in\Filt_{\T}\{E_i,E_{i+1},\ldots,E_n\}$, with $E_n'=E_n$,
such that $\bigoplus_{i=1}^nE_i'$ is tilting.
If the $E_i$ are vector bundles on a smooth projective variety,
then the $E_i'$ are vector bundles with filtrations by the indicated
$E_j$ in $\Coh X$.
\end{Prop}
\begin{proof}
For each $i$, set $E_i^{(i)}=E_i$. For $j=i+1,\ldots,n$, put
$V_{ij}:=\T(E_i^{(j-1)},E_j[1])$ and form a triangle
\[
    E_j\otimes_k V_{ij}^{*}\longrightarrow E_i^{(j)}
    \longrightarrow E_i^{(j-1)}
    \xrightarrow{\delta_{ij}} E_j[1]\otimes_k V_{ij}^{*},
\]
where $\delta_{ij}$ corresponds to the identity of $V_{ij}$.
Here tensoring an object with a finite-dimensional vector space
means taking the corresponding finite direct sum. This triangle
is the universal extension of $E_i^{(j-1)}$ by $E_j$.
Set $E_i'=E_i^{(n)}$.

Applying $\T(-,E_j[1])$ to the triangle shows that
$\T(E_i^{(j)},E_j[1])=0$: the preceding connecting map is
surjective by the choice of $\delta_{ij}$, and
$\T(E_j,E_j[1])=0$. For $\ell<j$, semiorthogonality gives
$\T(E_j,E_\ell[p])=0$ for every $p$, so previously obtained
vanishings against $E_\ell[1]$ are preserved. Induction therefore
gives $\T(E_i',E_\ell[1])=0$ for every $i,\ell$.
Since each $E_\ell'$ is an iterated extension of the original
objects, it follows that $\T(E_i',E_\ell'[1])=0$.
The vanishings in degrees $p<0$ and $p>1$ are preserved under
extensions in both variables, so $\bigoplus_iE_i'$ is pretilting.
Finally, the defining triangles, starting with $E_n'=E_n$ and
proceeding downwards in $i$, show that
$\thick\{E_1',\ldots,E_n'\}=\thick\{E_1,\ldots,E_n\}=\T$.
For vector bundles the triangles are extensions of vector bundles
in $\Coh X$, which proves the last assertion.
\end{proof}

\subsection{Slope functions on weak del Pezzo surfaces}\label{subsec:slopes}

Throughout this subsection, let $X$ denote a weak del Pezzo surface. We use the following fact which holds in arbitrary characteristic.

\begin{Prop}\label{irrmem}\cite[Proposition~2.12(2)]{BT}
Let $X$ be a weak del Pezzo surface. Then the linear system $|-K_X|$ has an irreducible reduced member.
\end{Prop}
In what follows, we pick and fix an irreducible reduced member $C\in|-K_X|$. Then by the adjunction formula, $C$ has arithmetic genus one.

\begin{Rem}
If $\ch k=0$, then it is well-known that we can take a smooth $C$. However, such $C$ does not always exist in positive characteristic (see \cite{KN}).
\end{Rem}

\begin{Lem}\label{pretilt}
Let $\E\in\vect X$ be a pretilting bundle. Then we have $\End_X(\E)\cong\End_C(\E|_C)$.
\end{Lem}
\begin{proof}
The short exact sequence
\[0\to\O_X(K_X)\to\O_X\to\O_C\to0\]
gives rise to an exact sequence
\[\Hom_X(\E,\E(K_X))\to\Hom_X(\E,\E)\to\Hom_X(\E,\E|_C)\to\Ext^1_X(\E,\E(K_X)).\]
Here, for $i=0,1$, we have
\[\Ext^i_X(\E,\E(K_X))\cong D\Ext^{2-i}_X(\E,\E)=0.\]
Since $\Hom_X(\E,\E|_C)\cong\Hom_C(\E|_C,\E|_C)$ holds, we get the desired isomorphism.
\end{proof}

For non-zero torsion-free sheaves $\F\in\Coh X$ and a nef and big divisor $D\in\Div X$, we consider the slope function
\[\mu_D(\F):=\frac{(c_1(\F)\cdot D)}{\rk(\F)}.\]
Similarly, for non-zero torsion-free sheaves $\F\in\Coh C$, we consider the slope function
\[\mu_C(\F):=\frac{\deg\F}{\rk(\F)}.\]
Note that $\mu_{-K_X}(\E)=\mu_C(\E|_C)$ holds for nonzero $\E\in\vect X$. Here, we see that \cite[Corollary~2.2.9]{Kul} holds in arbitrary characteristic and we include a complete proof for the convenience of the reader.

\begin{Prop}\label{stable}
For a nonzero pretilting bundle $\E\in\vect X$, the following conditions are equivalent.
\begin{enumerate}
\item $\E$ is $\mu_{-K_X}$-stable.
\item $\E$ is exceptional.
\item $\E|_C$ is $\mu_C$-stable.
\item $\End_C(\E|_C)=k$.
\end{enumerate}
\end{Prop}
\begin{proof}
(1)$\Rightarrow$(2) and (3)$\Rightarrow$(4) are well-known. (4)$\Rightarrow$(3) is due to \cite[Proposition~2.5(3)]{Dro}. By Lemma \ref{pretilt}, we have (2)$\Leftrightarrow$(4). In what follows, we prove (3)$\Rightarrow$(1).

Take a non-trivial epimorphism $f\colon\E\to\F$ where $\F\in\Coh X$ is a torsion-free sheaf. Then $\E':=\Ker f$ becomes locally free. We will prove $\mu_{-K_X}(\E)<\mu_{-K_X}(\F)$, or equivalently $\mu_{-K_X}(\E')<\mu_{-K_X}(\E)$. Let $U\subseteq X$ be a locally free locus of $\F$. Then the sequence $0\to\E'|_{U\cap C}\to\E|_{U\cap C}\to\F|_{U\cap C}\to0$ is exact. Here, $U$ must intersect with $C$ since $\dim(X\setminus U)=0$. Thus the morphism $\E'|_C\to\E|_C$ is injective at the generic point of $C$, which implies it is injective. Since $\E|_C$ is $\mu_C$-stable, we have
\[\mu_{-K_X}(\E')=\mu_C(\E'|_C)<\mu_C(\E|_C)=\mu_{-K_X}(\E).\qedhere\]
\end{proof}

For later use, we introduce the following category. For $\mu\in\mathbb{Q}$, let
\[\C_\mu:=\Filt\{\E\in\vect X\colon\text{exceptional bundle with }\mu_{-K_X}(\E)=\mu\}\subseteq\vect X\]
be the full subcategory of $\vect X$ consisting of vector bundles admitting a filtration whose factors are exceptional bundles of slope $\mu$. Then the following is an immediate corollary of Proposition \ref{stable}.

\begin{Cor}\label{extsemist}
Take $\mu\in\mathbb{Q}$ and a nonzero bundle $\E\in\C_\mu$.
\begin{enumerate}
\item $\E$ is $\mu_{-K_X}$-semistable with $\mu_{-K_X}(\E)=\mu$.
\item $\E|_C$ is $\mu_C$-semistable with $\mu_C(\E|_C)=\mu$.
\end{enumerate}
\end{Cor}

In Appendix~\ref{app:rigid-sheaves}, we will see that for any indecomposable rigid torsion-free sheaf $\E\in\Coh X$ with $\mu=\mu_{-K_X}(\E)$, we have $\E\in\C_\mu$ (Theorem \ref{rigtorfss}).

The following lemma will be used repeatedly and can be proved in the same way as \cite[Lemma~3.7]{KO} by using Corollary \ref{extsemist}(2).

\begin{Lem}\label{ext1}
Let $\mu_1, \mu_2\in\mathbb{Q}$ be rational numbers with $\mu_1<\mu_2$. Take $\E_i\in\C_{\mu_i}$ for $i=1,2$. If $\Ext_X^1(\E_2,\E_1)=0$ holds, then we have $\Ext_X^1(\E_1,\E_2)=0$.
\end{Lem}
\begin{proof}
If either bundle is zero, the assertion is immediate. We may therefore
assume that both bundles are nonzero.
Although it is completely parallel to \cite[Lemma~3.7]{KO}, we give a proof for the convenience of the reader.

Consider the short exact sequence
\[0\to\E_1(K_X)\to\E_1\to\E_1|_C\to0\]
from which we obtain an exact sequence
\[\Hom_X(\E_2,\E_1|_C)\to\Ext_X^1(\E_2,\E_1(K_X))\to\Ext_X^1(\E_2,\E_1).\]
From our assumption, we have $\Ext_X^1(\E_2,\E_1)=0$. Since $\mu_C(\E_1|_C)=\mu_1<\mu_2=\mu_C(\E_2|_C)$ holds and $\E_1|_C,\E_2|_C\in\vect C$ are $\mu_C$-semistable by Corollary \ref{extsemist}(2), we have
\[\Hom_X(\E_2,\E_1|_C)\cong\Hom_C(\E_2|_C,\E_1|_C)=0.\]
Therefore we obtain
\[\Ext_X^1(\E_1,\E_2)\cong D\Ext_X^1(\E_2,\E_1(K_X))=0.\qedhere\]
\end{proof}

We next record the mutations used in the construction. In
characteristic zero, the Ext-pair criterion is
\cite[Proposition~2.3.1, p.~2966]{Kul}, and the corresponding
universal extensions are described in
\cite[Section~3.1, Lemma-Definition~(2)]{Kul}.
We include the argument for our setting.

\begin{Prop}\label{mut}
Let $(\E_1,\E_2)$ be an exceptional pair of vector bundles on $X$
with $\mu_{-K_X}(\E_1)>\mu_{-K_X}(\E_2)$.
\begin{enumerate}
\item The left mutation $L_{\E_1}\E_2$ is an exceptional vector
bundle with
$\mu_{-K_X}(\E_2)<\mu_{-K_X}(L_{\E_1}\E_2)<\mu_{-K_X}(\E_1)$.
\item The right mutation $R_{\E_2}\E_1$ is an exceptional vector
bundle with
$\mu_{-K_X}(\E_2)<\mu_{-K_X}(R_{\E_2}\E_1)<\mu_{-K_X}(\E_1)$.
\end{enumerate}
These are Ext-type mutations: if $V:=\Ext_X^1(\E_1,\E_2)$,
our convention is that they are defined by the universal sequences
\[
    0\longrightarrow\E_2\longrightarrow L_{\E_1}\E_2
      \longrightarrow\E_1\otimes_k V\longrightarrow0,
\]
\[
    0\longrightarrow\E_2\otimes_k V^*
      \longrightarrow R_{\E_2}\E_1\longrightarrow\E_1
      \longrightarrow0.
\]
The mutated pairs $(L_{\E_1}\E_2,\E_1)$ and
$(\E_2,R_{\E_2}\E_1)$ are exceptional and generate the same
subcategory as $(\E_1,\E_2)$.
\end{Prop}
\begin{proof}
Semiorthogonality and Serre duality give
$\Ext_X^p(\E_1,\E_2(K_X))=0$ for every $p$. Restriction to $C$
therefore identifies
\[
    \Ext_X^p(\E_1,\E_2)
    \cong\Ext_C^p(\E_1|_C,\E_2|_C).
\]
By Proposition \ref{stable}, the restrictions are stable. Their
slopes give $\Hom_C(\E_1|_C,\E_2|_C)=0$.
Higher Ext groups vanish for vector bundles on the curve $C$.
Riemann--Roch on this integral curve of arithmetic genus one gives
\[
    \dim_k V=\rk(\E_1)\rk(\E_2)
       \bigl(\mu_{-K_X}(\E_1)-\mu_{-K_X}(\E_2)\bigr)>0.
\]
Thus the mutation triangles are exactly the two universal sequences
above. The usual long exact Hom sequences for these universal
extensions show that the mutated objects and pairs are exceptional;
the same sequences show that their generated subcategories agree.
The middle terms are vector bundles, and their slopes are strict
rank-weighted averages of the two original slopes.
\end{proof}

\section{Main results}\label{sec:main-results}

\subsection{Having a $d$-tilting bundle implies weak Fano}\label{subsec:necessity}

Throughout this subsection, $k$ is an arbitrary field. We prove the
following result for regular proper varieties, using the dualizing
sheaf specified in the conventions. Recall that weak Fano means that
$\omega_X^{-1}$ is nef and big; for a smooth projective surface over
an algebraically closed field, this is the weak del Pezzo condition.

\begin{Thm}\label{nece}
Let $X$ be a $d$-dimensional regular proper variety over $k$, with $d\geq1$, having a $d$-tilting bundle $\E\in\vect X$.
\begin{enumerate}
\item $\omega_X^{-1}$ is semiample and big. In particular, $X$ is weak Fano.
\item For $i>0$ and $n\ge0$, we have
\[H^i(X,\omega_X^{-n})=0.\]
\end{enumerate}
\end{Thm}

The proof first establishes semiampleness, then cohomology
vanishing, and finally bigness. It works in every characteristic
without assuming that $X$ is smooth or geometrically integral
over $k$.

Throughout this subsection, we use the notation in Theorem \ref{nece} and further fix the following notation. Let $T:=\Tot_X(\omega_X)$ and $\pi\colon T\to X$ the natural projection. Note that $\pi^*\E\in\vect T$ is a tilting bundle by Proposition \ref{dtiltsheaf}. Let
\[R:=\Gamma(T,\O_T)=\bigoplus_{n\ge0}\Gamma(X,\omega_X^{-n})\]
be the anti-canonical ring of $X$ and
\[\Pi:=\End_T(\pi^*\E)=\bigoplus_{n\ge0}\Hom_X(\E,\E\otimes\omega_X^{-n})\]
the $(d+1)$-preprojective algebra of the $d$-representation infinite algebra $\Pi_0=\End_X(\E)$. By Proposition \ref{dtiltsheaf}, we have a triangle equivalence
\[\RHom_T(\pi^*\E,-)\colon\D^b(\Coh T)\xrightarrow[\simeq]{}\per\Pi.\]

\begin{Prop}\label{prfsemiample}
$\omega_X^{-1}$ is semiample.
\end{Prop}
\begin{proof}
Suppose that $\omega_X^{-1}$ is not semiample. Then the closed subset
\[\mathbf{B}(\omega_X^{-1}):=\bigcap_{n>0}{\rm Bs}|\omega_X^{-n}|\subseteq X\]
is not empty. Take a closed point $x\in\mathbf{B}(\omega_X^{-1})$. Its residue field $\kappa(x)$ is a finite extension of $k$; it need not equal $k$. Let
\[C_x:=\pi^{-1}(x)\subseteq T,\]
which is isomorphic to $\mathbb{A}^1_{\kappa(x)}$. Consider the right $\Pi$-module
\[M:=\Hom_T(\pi^*\E,\O_{C_x}).\]
Put $r:=\rk\E$. Since
\[\Ext_T^{>0}(\pi^*\E,\O_{C_x})\cong\Ext_{C_x}^{>0}(\pi^*\E|_{C_x},\O_{C_x})\cong\Ext_{C_x}^{>0}(\O_{C_x}^r,\O_{C_x})=0\]
we have
\[M\cong\RHom_T(\pi^*\E,\O_{C_x})\in\per\Pi.\]
In particular, $M$ is finitely generated as a $\Pi$-module.

Put
\[
  I:=\Pi_{>0}:=\bigoplus_{n>0}\Pi_n.
\]
We first observe that every $u\in I$ acts nilpotently on
$\pi^*\E|_{C_x}$. Write
\[
  p_u(\lambda)
  =
  \lambda^r+c_1(u)\lambda^{r-1}+\cdots+c_r(u)\in R[\lambda]
\]
for the characteristic polynomial of $u$. Since $u$ has only
positive-degree components,
\[
  c_j(u)\in
  R_+:=\bigoplus_{n>0}H^0(X,\omega_X^{-n}).
\]
Since $x\in\mathbf{B}(\omega_X^{-1})$,
every element of $R_+$ restricts to zero on $C_x$. Hence
\[
  p_{u|_{C_x}}(\lambda)=\lambda^r,
\]
so $u|_{C_x}$ is nilpotent.

Choose a basis $t$ of $\omega_X^{-1}|_x$, so that
\[
    \omega_X^{-1}|_x=\kappa(x)t.
\]
This identifies
\[
    \Gamma(C_x,\O_{C_x})
    =\operatorname{Sym}_{\kappa(x)}(\omega_X^{-1}|_x)
    \cong\kappa(x)[t].
\]
Thus $C_x\cong\Spec\kappa(x)[t]$, and we also write $t$ for
the corresponding fiber coordinate. Let
\[
\varepsilon:\Pi\longrightarrow\End_{\kappa(x)}(\E|_x)
\]
be restriction to $C_x$ followed by specialization at $t=1$.
Under $\pi^*\E|_{C_x}\cong\E|_x\otimes_{\kappa(x)}\kappa(x)[t]$,
a homogeneous element $u\in\Pi_n$ restricts as
\[
    u|_{C_x}=t^n\varepsilon(u).
\]
Then
\(B:=\varepsilon(I)\subseteq\End_{\kappa(x)}(\E|_x)\) is a
finite-dimensional non-unital $k$-algebra, because
$[\kappa(x):k]<\infty$. Every element of $B$ is nilpotent by
the preceding argument. Thus $B$ is the Jacobson radical of its
finite-dimensional unitization $k\oplus B$. This radical is
nilpotent, so there is an integer
$N>0$ such that $B^n=0$ for every $n\geq N$.
For homogeneous $u_1,\ldots,u_N\in I$, of degrees
$n_1,\ldots,n_N$, their product restricts on $C_x$ to
\[
    t^{n_1+\cdots+n_N}
    \varepsilon(u_1)\cdots\varepsilon(u_N)=0.
\]
Expanding arbitrary elements into homogeneous components therefore gives
\[MI^N=0.\]

Since $\Pi$ is a tensor algebra, we have $I^N=\Pi_{\ge N}$. Hence the quotient
\[\Pi/I^N=\bigoplus_{n=0}^{N-1}\Pi_n\]
is finite-dimensional over $k$. Since $MI^N=0$, the finitely
generated $\Pi$-module $M$ is therefore finite-dimensional over $k$.
This contradicts the isomorphism
\[M\cong\kappa(x)[t]^r.\qedhere\]
\end{proof}

Next, we prove cohomology vanishing. In characteristic zero, the
normalized trace
\[
    \frac{1}{\rk\E}\tr:\sheafEnd_X(\E)\longrightarrow\O_X
\]
splits the scalar inclusion, as in \cite[Corollary~4.5]{Cha}.
The trace argument below works over every field, including when
the characteristic divides $\rk\E$: a relation in $K_0(X)$
supplies an endomorphism of trace $1$ without dividing by the rank.

\begin{Prop}\label{antcanvan}
For $i>0$ and $n\ge0$, we have
\[H^i(X,\omega_X^{-n})=0.\]
\end{Prop}
\begin{proof}
The tilting equivalence identifies $K^b(\add\E)$ with
$\D^b(\Coh X)$. Representing $\O_X$ by a bounded complex of
objects in $\add\E$ and taking its alternating class gives objects $\E_1,\E_2\in\add\E$ and integers $a_1,a_2\ge0$ such that
\[[\O_X]=a_1[\E_1]-a_2[\E_2]\]
holds in the Grothendieck group $K_0(X)$. Taking the rank, we have $1=a_1\rk(\E_1)-a_2\rk(\E_2)$.

Consider $\E':=\E_1\oplus\E_2\in\add\E$. Take idempotents $e_i:=(\E'\twoheadrightarrow\E_i\hookrightarrow\E')\in\End_X(\E')$ for $i=1,2$ and put $\theta:=a_1e_1-a_2e_2\in\End_X(\E')$. Then
\begin{equation}\label{tracetheta}
    \tr(\theta)=a_1\rk(\E_1)-a_2\rk(\E_2)=1
\end{equation}
holds in $\Gamma(X,\O_X)$. Consider the morphism
\[\sheafEnd_X(\E')\to\O_X;f\mapsto\tr(f\theta).\]
This gives a retraction of the scalar inclusion $\O_X\hookrightarrow\sheafEnd_X(\E')$ by \eqref{tracetheta}. Thus $\omega_X^{-n}$ is a direct summand of $\sheafEnd_X(\E')\otimes\omega_X^{-n}\cong\sheafHom_X(\E',\E'\otimes\omega_X^{-n}).$ Here, by Proposition \ref{dtiltsheaf}, we have
\[H^i(X,\sheafHom_X(\E',\E'\otimes\omega_X^{-n}))\cong\Ext_X^i(\E',\E'\otimes\omega_X^{-n})=0.\]
This gives the desired vanishing.
\end{proof}
\begin{Rem}
    This proof also shows the following general statement: if a noetherian integral scheme $Y$ has a tilting bundle, then $H^{>0}(Y,\O_Y)=0$ holds.
\end{Rem}

To give a transparent proof that $\omega_X^{-1}$ is big, we first prepare the following two lemmas.

\begin{Lem}\label{numgenample}
Let $Y$ be a projective variety of dimension $r$ over a field $k$, and
let $A$ be an ample line bundle on $Y$. Then there exists an integer
$e>0$ such that $A^{ej}$ is generated by $r+1$ global sections for
every $j>0$.
\end{Lem}

\begin{proof}
Replacing $A$ by a positive power, we may assume that $A$ is very
ample. Let
\[
    S=\bigoplus_{n\geq 0} S_n
\]
be the homogeneous coordinate ring of $Y$ with respect to the
embedding defined by $A$. Thus
\[
    Y=\operatorname{Proj}S,
    \qquad
    \mathcal{O}_Y(1)\cong A,
    \qquad
    \dim S=r+1.
\]

By graded Noether normalization, there exist homogeneous elements
$x_0,\ldots,x_r\in S$ such that $S$ is finite over
$k[x_0,\ldots,x_r]$. Replacing the $x_i$ by suitable positive powers,
we may assume that they all have the same degree, say $e$. The
extension
\[
    k[x_0,\ldots,x_r]\subset S
\]
remains finite after this replacement.

It follows that
\[
    S/(x_0,\ldots,x_r)
\]
is finite-dimensional over $k$. Hence, for some $N>0$,
\[
    S_+^N\subset (x_0,\ldots,x_r),
\]
where $S_+=\bigoplus_{n>0}S_n$ is the irrelevant ideal. Therefore
$x_0,\ldots,x_r$ have no common zero on
$\operatorname{Proj}S=Y$.

Since the $x_i$ all have degree $e$, they define global sections of
\[
    \mathcal{O}_Y(e)\cong A^e.
\]
Having no common zero is equivalent to these sections generating
$A^e$. Finally, for every $j>0$, the sections
\[
    x_0^j,\ldots,x_r^j
\]
still have no common zero, and hence generate
\[
    A^{ej}.
\]
This proves the claim.
\end{proof}

\begin{Lem}\label{semiamplevanimplybig}
Let $X$ be a $d$-dimensional proper variety over any field, with $d\geq1$, and let $\L$ be a semiample line bundle on it. If we have
\begin{equation}\label{semiample-negative-vanishing}
    H^i(X,\L^{-n})=0
\end{equation}
for every $0\le i<d$ and all $n\gg0$, then $\L$ is big.
\end{Lem}
\begin{proof}
Replacing $\L$ by its power, we may assume that $\L$ is basepoint-free, and $H^i(X,\L^{-n})=0$
for every $0\le i<d$ and $n>0$. Let $f\colon X\to Y$ be a proper surjective morphism to a projective variety $Y$ determined by $|\L|$. If we put $r:=\dim Y$, then it is enough to show $r=d$. Note that $f^*A\cong\L$ holds for an ample line bundle $A$ on $Y$.

Applying Lemma \ref{numgenample} to $A$, there exists $e>0$ such that $A^{ej}$ is
generated by $r+1$ global sections for every $j>0$. Thus
by replacing $\L$ by $\L^e$, we may assume that $\L^j$ is generated by $r+1$ global sections for every $j>0$. Choose global sections
$s_0,\ldots,s_r$ generating $\L$. Since they have no common
zero, their Koszul complex is exact:
\begin{equation}\label{semiample-koszul}
    0
    \longrightarrow \L^{-(r+1)}
    \longrightarrow
    \bigl(\L^{-r}\bigr)^{\binom{r+1}{r}}
    \longrightarrow \cdots
    \longrightarrow
    \bigl(\L^{-1}\bigr)^{r+1}
    \longrightarrow
    \O_X
    \longrightarrow 0.
\end{equation}

If $r=0$, exactness of \eqref{semiample-koszul} gives
$\L^{-1}\cong\O_X$, and hence
\[
    H^0(X,\L^{-1})
    \cong H^0(X,\O_X)\neq 0,
\]
contrary to \eqref{semiample-negative-vanishing}. Thus we may assume that $r>0$.

Suppose that $r<d$. Breaking \eqref{semiample-koszul} into short exact sequences and using
\[
    H^{j-1}(X,\L^{-j})=0
    \qquad (1\leq j\leq r),
\]
we obtain, by successive connecting homomorphisms, an injection
\[
    H^0(X,\O_X)
    \hookrightarrow
    H^r(X,\L^{-(r+1)}).
\]
Our assumption also gives
\[
    H^r(X,\L^{-(r+1)})=0.
\]
This is impossible because
$H^0(X,\O_X)$ is nonzero. Thus we obtain $r=d$.
\end{proof}

Combining these lemmas with Propositions \ref{prfsemiample} and \ref{antcanvan}, we prove that $\omega_X^{-1}$ is big.

\begin{Prop}\label{prfbig}
$\omega_X^{-1}$ is big.
\end{Prop}
\begin{proof}
By Proposition \ref{prfsemiample}, $\omega_X^{-1}$ is semiample. Serre duality and Proposition \ref{antcanvan} give, for $0\le i<d$ and $n>0$,
\[H^i(X,\omega_X^n)\cong DH^{d-i}(X,\omega_X^{-n+1})=0.\]
Therefore by Lemma \ref{semiamplevanimplybig}, $\omega_X^{-1}$ is big.
\end{proof}

Combining the above arguments, we can complete the proof of Theorem \ref{nece}.

\begin{proof}[Proof of Theorem \ref{nece}]
This follows from Propositions \ref{prfsemiample}, \ref{antcanvan},
and \ref{prfbig}.
\end{proof}

\subsection{Weak del Pezzo surfaces have $2$-tilting bundles}\label{subsec:sufficiency}

Throughout this subsection, $k$ is algebraically closed of arbitrary
characteristic. The goal is to prove the converse to Theorem
\ref{nece} for weak del Pezzo surfaces.

\begin{Thm}\label{suff}
Let $X$ be a weak del Pezzo surface over $k$. Then $X$ admits a $2$-tilting bundle $\E\in\vect X$. Moreover, we can take $\E$ so that it contains $\O_X$ as a direct summand.
\end{Thm}

We also record the consequence of Theorem \ref{nece} for the
rationality conjecture over algebraically closed fields.

\begin{Cor}\label{rat}
Let $X$ be a smooth projective surface over an algebraically closed
field having a $2$-tilting bundle. Then $X$ is rational.
\end{Cor}
\begin{proof}
Theorem \ref{nece} shows that $X$ is weak del Pezzo, and weak
del Pezzo surfaces over algebraically closed fields are rational.
\end{proof}

Throughout this subsection, let $X$ denote a weak del Pezzo surface. As a first step towards Theorem \ref{suff}, we construct a full exceptional collection satisfying slope inequalities.

\begin{Prop}\label{mutfec}
We have a full exceptional collection $(\E_1,\cdots ,\E_n)$ on $\D^b(\Coh X)$ satisfying the following conditions.
\begin{enumerate}
\item $\E_i\in\vect X$ for all $1\leq i\leq n$.
\item We have inequalities
\[\cdots\leq\mu_{-K_X}(\E_n(K_X))<\mu_{-K_X}(\E_1)\leq\mu_{-K_X}(\E_2)\leq\cdots\leq\mu_{-K_X}(\E_n)<\mu_{-K_X}(\E_1(-K_X))\leq\cdots.\]
\item $\E_a=\O_X$ for some $1\leq a\leq n$, and
\[\mu_{-K_X}(\E_1)=\cdots=\mu_{-K_X}(\E_a)=0<\mu_{-K_X}(\E_{a+1}).\]
\end{enumerate}
\end{Prop}
\begin{proof}
Write $d:=K_X^2>0$ and, in this proof only,
$\mu:=\mu_{-K_X}$. For $X=\mathbb P^1\times\mathbb P^1$, use
\[
    (\O_X,\O_X(1,0),\O_X(0,1),\O_X(1,1)),
\]
whose slopes are $0,2,2,4$, while $d=8$. For $X=\Sigma_2$, use
\[
    (\O_X,\O_X(F),\O_X(C_0+2F),\O_X(C_0+3F)),
\]
where $F$ is a fiber and $C_0^2=-2$. Its slopes are $0,2,4,6$,
and again $d=8$. These are standard tilting collections on
Hirzebruch surfaces; see \cite[Proposition~2.1]{HP14}.
For $\Sigma_2$, Chan \cite[Proposition~6.3 and the following remark]{Cha}
proves the $2$-tilting property and identifies this collection
with King's tilting bundle \cite{Kin97}.

In the remaining cases, $X$ is an iterated blowup of $\mathbb P^2$.
Let $H$ be the pullback of a line and let $e_1,\ldots,e_s$ be the
total transforms on $X$ of the successive exceptional divisors.
Starting from $(\O,\O(1),\O(2))$ on $\mathbb P^2$ and applying
Proposition \ref{excep} at each blowup gives explicitly
\[
    (\O_X,\O_X(e_s),\ldots,\O_X(e_1),\O_X(H),\O_X(2H)).
\]
Here $(-K_X)\cdot e_i=1$ and $(-K_X)\cdot H=3$, also for
infinitely near blowup centers. Hence every term other than
$\O_X$ has positive slope.

We explain how to order such a collection using only the right
mutations of Proposition \ref{mut}(2). Choose the leftmost term
of minimal slope and move it to the first position. Each preceding
term has strictly larger slope, so every required mutation is of
Ext type. Fix this first term and repeat on the remaining terms.
Induction on the length gives a nondecreasing sequence of slopes
in finitely many mutations. In each step $(A,B)$ with
$\mu(A)>\mu(B)$ is replaced by $(B,R_BA)$, and
\[
    \mu(B)<\mu(R_BA)<\mu(A).
\]
Thus all slopes stay nonnegative and their sum decreases.
The term $\O_X$ is preserved; moreover, it is the last term of
slope zero, since equal-slope terms are never interchanged.

Let $(\E_1,\ldots,\E_n)$ now be ordered in this way. If
$\mu(\E_n)<d$, the required inequalities hold. Otherwise use
Serre duality to rotate the full exceptional collection to
\[
    (\E_n(K_X),\E_1,\ldots,\E_{n-1}),
\]
and order it again by the preceding right mutations.
The new term has slope $\mu(\E_n)-d\geq0$. Rotation decreases
the sum of the slopes by $d$, and reordering can only decrease
it further. The term $\O_X$ remains the last zero-slope term,
including when the rotated term has slope zero. Since the sum
of the slopes is nonnegative, only finitely many such rotations
are possible. We therefore reach
\[
    0=\mu(\E_1)\leq\cdots\leq\mu(\E_n)<d.
\]
The term $\O_X$ occurs at the end of the zero-slope block.
There is a positive-slope term: otherwise fullness would force
every class in $K_0(X)$ to have
$(-K_X)\cdot c_1=0$, contrary to the class of $\O_X(-K_X)$.
This proves all the assertions.
\end{proof}

As in \cite[Proposition~7.3]{VdB04a}, if $X$ is del Pezzo, then the full exceptional collection obtained in Proposition \ref{mutfec} is strong, and its direct sum is a $2$-tilting bundle by Theorem \ref{2tiltslope}(1). In the general case, we need a second step: applying iterated universal extensions to this full exceptional collection.

\begin{proof}[Proof of Theorem \ref{suff}]
Take a full exceptional collection $(\E_1,\cdots ,\E_n)$ as in Proposition \ref{mutfec}. Take $1\leq i<i'\leq n$. Since $\mu_{-K_X}(\E_{i'})\geq\mu_{-K_X}(\E_i)>\mu_{-K_X}(\E_i(K_X))$, by Proposition \ref{stable}, we have 
\[\Ext_X^2(\E_i,\E_{i'})\cong D\Hom_X(\E_{i'},\E_i(K_X))=0.\]
Take $0=a_0<a_1<\cdots<a_m=n$ satisfying
\begin{align*}
&\mu_1:=\mu_{-K_X}(\E_{a_0+1})=\cdots=\mu_{-K_X}(\E_{a_1}) \\
&<\mu_2:=\mu_{-K_X}(\E_{a_1+1})=\cdots=\mu_{-K_X}(\E_{a_2}) \\
&<\cdots \\
&<\mu_m:=\mu_{-K_X}(\E_{a_{m-1}+1})=\cdots=\mu_{-K_X}(\E_{a_m}).
\end{align*}
If there exists $1\leq j\leq m$ with $i\leq a_j<i'$, then by Lemma \ref{ext1}, we have $\Ext_X^1(\E_i,\E_{i'})=0$. Therefore only $\Ext_X^1(\E_i,\E_{i'})$ may not vanish where $a_{j-1}<i<i'\leq a_j$ for some $1\leq j\leq m$.

By Proposition \ref{uniext}, for each $1\le j\le m$, there exists a vector bundle $\F_j\in\Filt\{\E_i\mid a_{j-1}<i\le a_j\}\subseteq\C_{\mu_j}$ such that $\F_j\in\thick\{\E_i\mid a_{j-1}<i\le a_j\}$ is a tilting object. By the previous argument, we have $\Ext_X^{>0}(\F_j,\F_{j'})=0$ for $j\neq j'$. Thus $\E:=\bigoplus_{j=1}^m\F_j$ is a tilting bundle. By the construction of Proposition \ref{mutfec}, we have $\E_{a_1}=\O_X$. Thus $\F_1$ has $\O_X$ as a direct summand by its construction.

Take $1\le j,j'\le m$. Observe that $\Ext^1_X(\F_{j'}(-K_X),\F_j)\cong D\Ext_X^1(\F_j,\F_{j'})=0$ holds. Since $\mu_{-K_X}(\F_{j'}(-K_X))=\mu_{j'}+K_X^2>\mu_j=\mu_{-K_X}(\F_j)$, by Lemma \ref{ext1}, we have
\[\Ext^1_X(\F_j,\F_{j'}(-K_X))=0.\]
Since $\mu_{-K_X}(\F_{j'}(-K_X))=\mu_{j'}+K_X^2>\mu_j-K_X^2=\mu_{-K_X}(\F_j(K_X))$, by Corollary \ref{extsemist}, we have
\[\Ext^2_X(\F_j,\F_{j'}(-K_X))\cong D\Hom_X(\F_{j'}(-K_X),\F_j(K_X))=0.\]
Thus $\E$ is a $2$-tilting bundle by Proposition \ref{dtiltsheaf}.
\end{proof}

\section{Applications to non-commutative crepant resolutions}\label{sec:nccr}

We first construct NCCRs of anti-canonical rings of regular proper
varieties admitting $d$-tilting bundles. We then apply this construction
to anti-canonical cones over Du Val del Pezzo surfaces.

\begin{Def}\cite[Definition~4.1]{VdB04a}
Let $R$ be a noetherian Gorenstein normal domain with $\dim R\geq2$.
For a non-zero finitely generated reflexive $R$-module $M$, the algebra
$\End_R(M)$ is called a {\it non-commutative crepant resolution (NCCR)}
of $R$ if it is maximal Cohen--Macaulay as an $R$-module and
$\gl\End_R(M)_{\mathfrak p}<\infty$ for every
$\mathfrak p\in\Spec R$.
\end{Def}

\subsection{Anti-canonical rings and $d$-tilting bundles}\label{subsec:anticanonical-nccr}

Throughout this subsection, $k$ is an arbitrary field and
$\omega_X$ is the dualizing sheaf from the conventions.
We use ``crepant resolution'' for a proper birational morphism with
regular source whose dualizing sheaf is the canonical pullback
of that of the target. Smoothness over $k$ is not required.
For an $R$-module $P$, write $P^*:=\Hom_R(P,R)$.

Over an algebraically closed field of characteristic zero, the Fano
case of the following cone construction is recalled in
\cite[Lemma~7.1(2),(3)]{VdB04a}, and its NCCR application is given in
\cite[Proposition~7.2]{VdB04a}.
We need a version over arbitrary fields with $\omega_X^{-1}$ semiample
and big, as supplied by Theorem \ref{nece}.

\Needspace{16\baselineskip}
\begin{Lem}\label{anticanonical-contraction}
Let $X$ be a $d$-dimensional regular proper variety with $d\geq1$.
Assume that $\omega_X^{-1}$ is semiample and big and that
\[
    H^i(X,\omega_X^{-n})=0\qquad(i>0,\ n\geq0).
\]
Put
\[
    \pi:T:=\Tot_X(\omega_X)\longrightarrow X,
    \qquad
    R:=\Gamma(T,\O_T)
      =\bigoplus_{n\geq0}H^0(X,\omega_X^{-n}).
\]
Then $R$ is a finitely generated Gorenstein normal domain of dimension
$d+1$. The natural morphism $\rho:T\to Z:=\Spec R$ is a proper
crepant resolution and satisfies
\[
    \mathbb R\rho_*\O_T\cong\O_Z,
    \qquad
    \rho^!\O_Z\cong\O_T.
\]
\end{Lem}
\begin{proof}
Since $T$ is regular and integral, $R=\Gamma(T,\O_T)$ is a normal
domain. Choose $m>0$ such that $\omega_X^{-m}$ is globally generated,
and let $f:X\to Y$ be the proper surjective morphism to its projective
image defined by $|\omega_X^{-m}|$. Thus
$f^*\O_Y(1)\cong\omega_X^{-m}$.
The section ring $B:=\bigoplus_{n\geq0}H^0(Y,\O_Y(n))$ is finitely
generated over $k$. By the projection formula,
\[
    R\cong\bigoplus_{i=0}^{m-1}\bigoplus_{n\geq0}
        H^0\bigl(Y,f_*(\omega_X^{-i})\otimes\O_Y(n)\bigr)
\]
as $B$-modules. Each inner sum is a finite $B$-module by Serre's
theorem, so $R$ is a finitely generated $k$-algebra.

Put $S:=R^{(m)}$. Evaluation gives a surjection
\[
    S\otimes_k\O_X\longrightarrow\bigoplus_{n\geq0}\omega_X^{-mn},
\]
so $T_m:=\Tot_X(\omega_X^m)$ is a closed subscheme of
$X\times_k\Spec S$ and is proper over $\Spec S$.
The fiberwise $m$-th power map $T\to T_m$ is finite. Hence
$T\to\Spec S$ is proper, and its factorization through the finite
morphism $\Spec R\to\Spec S$ shows that $\rho$ is proper.
This argument requires neither projectivity of $X$ nor generation of
$R$ in degree one.

Bigness of $\omega_X^{-1}$ means that its Iitaka dimension is $d$,
so its section ring satisfies
\[
    \dim R=d+1=\dim T.
\]
The inclusion $R\subset k(T)$ shows that $\rho$ is dominant, and the
dimension equality makes it generically finite.
On the other hand, since $Z$ is affine and $R=\Gamma(T,\O_T)$,
we have $\rho_*\O_T=\O_Z$.
Let $\eta$ be the generic point of $Z$. The generic fiber $T_\eta$
is proper, integral and zero-dimensional, hence $T_\eta=\Spec k(T)$.
Taking the stalk at $\eta$ therefore gives
\[
    k(Z)=(\rho_*\O_T)_\eta
        =H^0(T_\eta,\O_{T_\eta})=k(T).
\]
Thus $\rho$ is birational.
The assumed vanishing and affineness of $\pi$ give
\[
    H^i(T,\O_T)=\bigoplus_{n\geq0}H^i(X,\omega_X^{-n})=0\qquad(i>0),
\]
whence $\mathbb R\rho_*\O_T\cong\O_Z$.

Since $\pi:T\to X$ is smooth of relative dimension one and $\Omega^1_{T/X}\cong\pi^*\omega_X^{-1}$, smooth duality gives
\[
\omega_T\cong\pi^*\omega_X\otimes\Omega^1_{T/X}
\cong\mathcal O_T.
\]

Let $D_Z$ be the dualizing complex of $Z$ over $k$.
Since $\mathbb R\rho_*\O_T\cong\O_Z$ and
$\omega_T\cong\O_T$, Grothendieck duality
\cite[Tag~0AUE]{stacks-project} gives
\[
D_Z
\cong \mathbb R\rho_*D_T
\cong \mathbb R\rho_*\O_T[d+1]
\cong \O_Z[d+1].
\]
Thus $Z$ is Gorenstein with $\omega_Z\cong\O_Z$.
Moreover,
\[
\rho^!\O_Z
\cong \rho^!D_Z[-d-1]
\cong D_T[-d-1]
\cong \O_T.
\]
The duality trace identifies
$\omega_Z\cong\rho_*\omega_T$.
Under this identification, the natural map
$\rho^*\omega_Z\to\omega_T$ is the evaluation map,
which is an isomorphism since
$\omega_T\cong\O_T$ and
$\rho_*\O_T=\O_Z$.
Hence $\rho$ is crepant.
\end{proof}

Since $X$ is not assumed to be geometrically integral over $k$,
the degree-zero part $R_0=H^0(X,\O_X)$ need not equal $k$.
Properness and integrality imply that $R_0$ is a finite field
extension of $k$. This presents no essential difficulty: one may
regard $X$ as a variety over $R_0$ and replace the base field $k$
by $R_0$, so that $R_0=k$, without changing the underlying scheme
or the hypotheses used here.

\Needspace{14\baselineskip}
\begin{Thm}\label{NCCRanticanonical}
Let $X$ be a $d$-dimensional regular proper variety over an arbitrary
field, with $d\geq1$, and let $\E\in\vect X$ be a $d$-tilting bundle.
Put
\[
    \pi:T:=\Tot_X(\omega_X)\longrightarrow X,
    \qquad
    R:=\bigoplus_{n\geq0}H^0(X,\omega_X^{-n}),
    \qquad
    M:=\Gamma(T,\pi^*\E)^{**}.
\]
Then $R$ is a finitely generated Gorenstein normal domain and
\[
    \End_T(\pi^*\E)
    \ \cong\ \End_R(M)
    \ \cong\ \Pi_{d+1}(\End_X(\E))
\]
is an NCCR of $R$.
If $\E$ contains $\O_X$ as a direct summand, then
$M=\Gamma(T,\pi^*\E)$ is maximal Cohen--Macaulay and contains $R$
as a direct summand.
\end{Thm}
\begin{proof}
We follow the argument of Iyama--Wemyss
\cite[Corollary~4.15]{IW14b}. Their result is stated for projective
morphisms; we give the proper version for our regular source,
including the identification of the specified module $M$.

Theorem \ref{nece} gives the
hypotheses of Lemma \ref{anticanonical-contraction}. Thus
$\rho:T\to\Spec R$ is a proper crepant resolution and
$\rho^!\O_{\Spec R}\cong\O_T$.
Set $\V:=\pi^*\E$ and $\Pi:=\End_T(\V)$.
By Proposition \ref{dtiltsheaf}, $\V$ is tilting and
$\Pi\cong\Pi_{d+1}(\End_X(\E))$.

Since $\rho$ is proper, $\rho_*\V$ and
$\rho_*\sheafEnd_T(\V)$ are coherent on $\Spec R$.
Consequently $P:=\Gamma(T,\V)$ and $\Pi$ are finite $R$-modules.
The scalar action gives a central inclusion $R\hookrightarrow\Pi$.
For Cohen--Macaulayness, we use the duality argument in
\cite[Lemma~4.8]{IW14b}, which applies to the proper morphism $\rho$
by Grothendieck duality \cite[Tag~0AUE]{stacks-project}.
Since $\V$ is tilting,
$\mathbb R\Gamma(T,\sheafEnd_T(\V))\cong\Pi$.
The perfect trace pairing gives
$\sheafEnd_T(\V)^\vee\cong\sheafEnd_T(\V)$ in every characteristic.
Using $\rho^!\O_{\Spec R}\cong\O_T$, we therefore obtain
\[
    \mathbb R\Hom_R(\Pi,R)
    \cong\mathbb R\Gamma(T,\sheafEnd_T(\V)^\vee)
    \cong\Pi.
\]
Hence $\Ext_R^{>0}(\Pi,R)=0$.
Since $R$ is Gorenstein, $\Pi$ is maximal Cohen--Macaulay, and in
particular reflexive as an $R$-module.

To identify the endomorphism ring, we follow the codimension-one
argument in \cite[Theorem~4.5]{IW14b}.
The action on $P$ and functoriality of the double dual give a natural
homomorphism
\[
    \Pi\longrightarrow\End_R(P^{**})=\End_R(M).
\]
The module $M$ and its endomorphism module are reflexive.
A proper birational morphism to a normal variety is an isomorphism
near every codimension-one point of the target. Thus, after
localization at any height-one prime of $R$, the displayed
homomorphism is the endomorphism-ring identification for the same
vector bundle. It is therefore an isomorphism in codimension one.
Reflexivity of both sides now gives $\Pi\cong\End_R(M)$ globally.

Finally, $\Coh T$ is a noetherian abelian category of finite global
dimension, since $T$ is regular, noetherian and of finite dimension. Thus \cite[Theorem~1]{KK20} gives
\[
    \gl\Pi\leq 2\gl(\Coh T)<\infty.
\]
Hence $\Pi$ is an NCCR.

If $\O_X$ is a direct summand of $\E$, then $\O_T$ is a direct
summand of $\V$. Consequently
$\V=\sheafHom_T(\O_T,\V)$ is a direct summand of
$\sheafEnd_T(\V)$, so $P$ is a direct summand of $\Pi$.
Thus $P$ is maximal Cohen--Macaulay and $P=P^{**}=M$.
The summand $\O_T$ of $\V$ also gives $R$ as a direct summand of $M$.
\end{proof}

\begin{Cor}\label{RPifin}
Let $X$ be a $d$-dimensional regular proper variety over any field,
with $d\geq1$, and let $\E$ be a $d$-tilting bundle.
Then $\End_X(\E)$ is $d$-representation tame.
\end{Cor}
\begin{proof}
By Proposition \ref{dtiltsheaf}, $\End_X(\E)$ is $d$-representation
infinite. Theorem \ref{NCCRanticanonical} shows that its
$(d+1)$-preprojective algebra is module-finite over the commutative
noetherian anti-canonical ring $R$, which acts centrally.
Thus $\End_X(\E)$ is $d$-representation tame.
\end{proof}

\begin{Rem}
The construction requires neither that $R$ be generated in degree one
nor that $\E$ contain a line bundle as a direct summand.
The algebra is naturally non-negatively graded:
\[
    \End_T(\pi^*\E)
    =\bigoplus_{n\geq0}
        \Hom_X(\E,\E\otimes\omega_X^{-n}).
\]
If $X$ is projective, the morphism $\rho$ is projective as well.
For regular varieties over imperfect fields, $\omega_X$ must be
interpreted as the dualizing sheaf, rather than the top exterior
power of K\"ahler differentials.
\end{Rem}

\subsection{Du Val del Pezzo cones}\label{subsec:du-val-cones}

Throughout this subsection, $k$ is algebraically closed of arbitrary
characteristic.

\begin{Def}
A normal projective surface $X$ is called {\it Du Val del Pezzo} if
$-K_X$ is ample and $X$ has at worst Du Val singularities.
\end{Def}

Van den Bergh constructed NCCRs of anti-canonical cones over smooth
del Pezzo surfaces in \cite[Proposition~7.3]{VdB04a}.
Combining Theorem \ref{NCCRanticanonical} with Theorem \ref{suff}
extends this construction to Du Val del Pezzo surfaces.

\Needspace{16\baselineskip}
\begin{Thm}\label{NCCRsingdP}
Let $X$ be a Du Val del Pezzo surface over $k$ and put
\[
    R:=\bigoplus_{n\geq0}H^0(X,\omega_X^{-n}).
\]
Then the anti-canonical cone $\Spec R$ admits an NCCR.
More precisely, if $f:\widetilde X\to X$ is the minimal resolution,
there is a $2$-tilting bundle $\E$ on $\widetilde X$ such that
\[
    M:=\bigoplus_{n\geq0}
        H^0(\widetilde X,\E\otimes\omega_{\widetilde X}^{-n})
\]
is a graded maximal Cohen--Macaulay $R$-module containing $R$ as a
graded direct summand, and
\[
    \End_R(M)
    \cong\End_{\Tot_{\widetilde X}(\omega_{\widetilde X})}(\pi^*\E)
    \cong\Pi_3(\End_{\widetilde X}(\E))
\]
is a non-negatively graded NCCR of $R$.
\end{Thm}
\begin{proof}
Since Du Val singularities are Gorenstein and their minimal resolutions
are crepant, we have
\[
    \omega_{\widetilde X}\cong f^*\omega_X,
    \qquad f_*\O_{\widetilde X}=\O_X.
\]
Thus $-K_{\widetilde X}=f^*(-K_X)$ is nef and big, so
$\widetilde X$ is a weak del Pezzo surface.
The projection formula identifies the anti-canonical rings:
\[
    \bigoplus_{n\geq0}H^0(\widetilde X,
        \omega_{\widetilde X}^{-n})
    \cong\bigoplus_{n\geq0}H^0(X,\omega_X^{-n})=R.
\]
By Theorem \ref{suff}, $\widetilde X$ has a $2$-tilting bundle
$\E$ containing $\O_{\widetilde X}$ as a direct summand.
Apply Theorem \ref{NCCRanticanonical} to $\widetilde X$ and $\E$.
It gives the required NCCR and shows that
\[
    M=\Gamma(\Tot_{\widetilde X}(\omega_{\widetilde X}),\pi^*\E)
\]
is maximal Cohen--Macaulay, so no double dual is needed in this case.
The displayed grading of $M$ comes from the affine projection $\pi$,
and the grading of its endomorphism algebra comes from the formula
\[
    \End_R(M)
    \cong\bigoplus_{n\geq0}
        \Hom_{\widetilde X}
        (\E,\E\otimes\omega_{\widetilde X}^{-n}).
\]
Finally, the summand $\O_{\widetilde X}$ of $\E$ gives $R$ as a
direct summand in the indicated grading.
\end{proof}

\begin{Ex}
Let $f\in S:=k[x,y,z,w]$ be a homogeneous polynomial of degree $3$ and $R:=S/(f)$. If $X:=(f=0)\subseteq\mathbb{P}^3$ is smooth, then it is known that $R$ has an NCCR by \cite[Proposition~7.3]{VdB04a}. Our results imply that even when $X$ is singular, if $\Sing X$ consists only of Du Val singularities, then $R$ has an NCCR. For example, if $f=xyz+w^3$, then $\Sing X=\{[1:0:0:0],[0:1:0:0],[0:0:1:0]\}$ consists of three Du Val singularities of type $A_2$, so this $R$ admits an NCCR.
\end{Ex}

\begin{appendix}

\section{Structure of rigid torsion-free sheaves}\label{app:rigid-sheaves}

Throughout this appendix, $k$ is algebraically closed of arbitrary
characteristic. We prove the following structure theorem for weak
del Pezzo surfaces over $k$.

\begin{Thm}\label{rigtorfss}
Let $X$ be a weak del Pezzo surface and $\F\in\Coh X$ an indecomposable rigid torsion-free sheaf with $\mu:=\mu_{-K_X}(\F)$. Then we have
\[\F\in\C_\mu.\]
In particular, $\F$ is pretilting and $\mu_{-K_X}$-semistable (see Corollary \ref{extsemist}).
\end{Thm}

This is known when $\ch k=0$ \cite[Theorem~2.4.1]{Kul}. Our proof is essentially the same as \cite{Kul}, but we have to be careful that a smooth anti-canonical curve need not exist in positive characteristic \cite{KN}. In addition, our proof does not use spectral sequences.

We use the following form of Mukai's lemma repeatedly. The sheaf
version is \cite[Lemma~2.2(2a)]{KO}; the elementary argument below
works in a Hom-finite triangulated category.

\begin{Lem}\label{heredrig}
Let $\T$ be a Hom-finite triangulated category and
\[
    X\xrightarrow{a}Y\xrightarrow{b}Z\xrightarrow{c}X[1]
\]
an exact triangle. If $\T(X,Z)=0$ and $\T(Z,X[2])=0$, then
\[
    \dim_k\T(Y,Y[1])\geq
    \dim_k\T(X,X[1])+\dim_k\T(Z,Z[1]).
\]
\end{Lem}
\begin{proof}
Let $W\subseteq\T(Y,Y[1])$ consist of the middle components of
morphisms from the displayed triangle to its shift by $1$.
For each $h\in W$, its components
$u\colon X\to X[1]$ and $v\colon Z\to Z[1]$ are unique:
the long exact Hom sequences and $\T(X,Z)=0$ make the maps
\[
    \T(X,X[1])\xrightarrow{a[1]\circ-}\T(X,Y[1]),
    \qquad
    \T(Z,Z[1])\xrightarrow{-\circ b}\T(Y,Z[1])
\]
injective. Thus $h\mapsto(u,v)$ defines a linear map
$W\to\T(X,X[1])\oplus\T(Z,Z[1])$.
It is surjective. Indeed, any pair $(u,v)$ satisfies the
compatibility condition on the connecting morphisms, since both
composites lie in $\T(Z,X[2])=0$. The morphism-of-triangles
axiom, applied after rotation, then supplies a middle component
$h$. Taking dimensions proves the inequality.
\end{proof}

We also need a different consequence of Mukai's results, namely
local freeness of rigid torsion-free sheaves. Its proof uses
\cite[Lemma~2.2(1)]{KO}, rather than Lemma \ref{heredrig}.

\begin{Cor}\label{rigtorfbdl}\cite[Corollary~2.3]{KO}
Let $X$ be a smooth projective surface. Every rigid torsion-free
sheaf on $X$ is a vector bundle.
\end{Cor}
\begin{proof}
For a torsion-free sheaf $\F$, Mukai's inequality
\cite[Lemma~2.2(1)]{KO} gives
\[
    \dim_k\Ext_X^1(\F,\F)\geq
    \dim_k\Ext_X^1(\F^{**},\F^{**})
       +2\operatorname{length}(\F^{**}/\F).
\]
The argument for this inequality is local on the smooth surface
and uses Serre duality, so it is valid in arbitrary characteristic.
If $\F$ is rigid, the quotient has length zero. Hence
$\F=\F^{**}$, which is locally free on a smooth surface.
\end{proof}

In what follows, let $X$ be a weak del Pezzo surface. We introduce another slope function on $X$. Take and fix an ample divisor $A\in\Div X$. For non-zero torsion-free sheaves $\F\in\Coh X$, define
\[\gamma(\F):=\Big(\mu_{-K_X}(\F),\mu_A(\F), \frac{c_1(\F)^2-2c_2(\F)}{\rk(\F)}\Big)\in\mathbb{Q}^3.\]
We equip $\mathbb Q^3$ with the lexicographical order. Each component
of $\gamma$ is a rank-normalized additive invariant; in particular,
its third numerator is $2\operatorname{ch}_2$. Thus slopes satisfy
the rank-weighted see-saw property in short exact sequences of
torsion-free sheaves. We first record the following properties.

\begin{Prop}\label{gammaprop}
Let $\E,\F\in\Coh X$ be non-zero torsion-free sheaves.
\begin{enumerate}
\item Assume we have an inclusion $\E\subsetneq\F$ and they have the same rank. Then we have
\[\gamma(\E)<\gamma(\F).\]
\item Assume $\E$ is $\gamma$-semistable and $\F$ is $\gamma$-stable and $\gamma(\E)=\gamma(\F)$. Then any non-zero homomorphism $\E\to\F$ is an epimorphism.
\end{enumerate}
\end{Prop}
\begin{proof}
(1) It is easy to see that $\mu_{-K_X}(\E)\le\mu_{-K_X}(\F)$ and $\mu_A(\E)\le\mu_A(\F)$ hold. Assume that we have $\mu_{-K_X}(\E)=\mu_{-K_X}(\F)$ and $\mu_A(\E)=\mu_A(\F)$. If we put $\G:=\F/\E\neq0$, then $(c_1(\G)\cdot A)=0$ holds by our assumption. Since $A$ is ample, this implies $\dim\G=0$. Thus for any $n\ge0$, we have $\chi(\G\otimes\omega_X^{-n})>0$ and  $\chi(\E\otimes\omega_X^{-n})<\chi(\F\otimes\omega_X^{-n})$. Therefore we obtain $c_1(\E)^2-2c_2(\E)<c_1(\F)^2-2c_2(\F)$.

(2) This is an immediate corollary of (1).
\end{proof}

The proof of Lemma \ref{gammafiltrations} is essentially the same
as in \cite{Kul}. For the reader's convenience, we include a sketch
to make clear that no restriction on the characteristic is needed.

\begin{Lem}[{cf.\ \cite[Propositions~1.2.3 and 1.3.6]{Kul}}]\label{gammafiltrations}
Every nonzero torsion-free sheaf on $X$ has a Harder--Narasimhan
filtration for $\mu_{-K_X}$ and for $\gamma$. Every
$\gamma$-semistable torsion-free sheaf has a Jordan--H\"older
filtration with $\gamma$-stable factors of the same $\gamma$-slope.
\end{Lem}
\begin{proof}
Put $H:=-K_X$ and fix a nonzero torsion-free sheaf $\F$.
The globally generated twist argument of \cite[Lemma~1.2.1]{Kul}
gives lower bounds for the $\mu_H$- and $\mu_A$-slopes of its
nonzero torsion-free quotients. Their ranks are at most $\rk(\F)$,
so these slopes have bounded denominators. We can therefore minimize
$\mu_H$ and then $\mu_A$ among the quotients attaining the first minimum.

For these quotients, fixing $\mu_A$ fixes the normalized Hilbert
slope for a very ample multiple of $A$. Grothendieck's lemma
\cite[Lemma~1.7.9 and Remark~1.7.10]{HL97} thus gives only finitely
many Hilbert polynomials. Surface Riemann--Roch yields
\[
    \frac{c_1(\mathcal{Q})^2-2c_2(\mathcal{Q})}{\rk(\mathcal{Q})}
    =\frac{2\chi(\mathcal{Q})}{\rk(\mathcal{Q})}-2\chi(\O_X)-\mu_H(\mathcal{Q}),
\]
so the third component also attains a minimum. Hence minimum-slope
quotients exist for both $\mu_H$ and $\gamma$.

The construction in the proof of \cite[Proposition~1.2.3]{Kul}
now gives the Harder--Narasimhan filtrations. The Jordan--H\"older
construction of \cite[Proposition~1.3.6 and the proof of Proposition~1.3.7]{Kul}
applies by Proposition \ref{gammaprop}(1).
The cited arguments use global generation, the see-saw property
and induction on rank, and remain valid in arbitrary characteristic.
\end{proof}

The following is a technical lemma which supports Lemma \ref{ssCmu}.

\begin{Lem}\label{compfacrig}
Let $\T$ be a triangulated category. Let $E\in\T$ be an object with $\T(E,E[2])=0$. Assume there exists an object $0\neq F\in\Filt_{\T}\{E\}$ satisfying $\T(F,F[1])=0$. Then we have
\[\T(E,E[1])=0.\]
In particular, $\Filt_{\T}(\add E)=\add E$ holds.
\end{Lem}
\begin{proof}
There exist two exact triangles
\[E\xrightarrow{a}F\to G\dashrightarrow\text{ and }G'\to F\xrightarrow{b}E\dashrightarrow\]
where $G,G'\in\Filt_{\T}\{E\}$. Take a morphism $f\colon E\to E[1]$. Since $\T(G[-1],E[1])=0$ holds, there exists $g\colon F\to E[1]$ with $ga=f$. Since $\T(F,G'[2])=0$, there exists $h\colon F\to F[1]$ with $b[1]h=g$. Here, by our assumption, we have $h=0$. Thus we obtain
\[f=b[1]ha=0.\qedhere\]
\end{proof}

The following lemma follows the arguments of
\cite[Lemmas~2.4.3 and 2.4.4]{Kul}.

\begin{Lem}\label{ssCmu}
Let $\F\in\Coh X$ be a nonzero $\mu_{-K_X}$-semistable rigid torsion-free sheaf with $\mu:=\mu_{-K_X}(\F)$. Then we have $\F\in\C_{\mu}$.
\end{Lem}
\begin{proof}
Take the Harder--Narasimhan filtration from Lemma
\ref{gammafiltrations} with respect to $\gamma$, written as $\F\in\G_n*\cdots*\G_1$ using the
extension notation of Subsection~\ref{subsec:exceptional-sequences}.
Its factors satisfy $\gamma(\G_n)>\cdots>\gamma(\G_1)$.
Their $\mu_{-K_X}$-slopes are nonincreasing, and semistability
of $\F$ forces all of them to equal $\mu$.
For each step of this filtration, the Hom group from the subobject
to the quotient is zero by $\gamma$-semistability. The reverse
$\Ext^2$ group is zero by Serre duality, since twisting by $K_X$
lowers the common $\mu_{-K_X}$-slope by $K_X^2>0$.
Lemma \ref{heredrig} therefore makes all the factors rigid.
It suffices to prove the assertion for each of these factors,
so we may assume that $\F$ is $\gamma$-semistable.

The isotypic quotient construction of
\cite[Proposition~1.3.7]{Kul} applies to $\gamma$ as follows.
Choose a $\gamma$-stable quotient of $\F$ of the same slope.
As long as the remaining kernel maps nontrivially to that quotient,
the map is surjective by Proposition \ref{gammaprop}(2); replace
the kernel by its kernel. The rank decreases at each step.
Repeating on the final kernel gives $\gamma$-stable torsion-free sheaves $\E_1,\cdots,\E_n$ with $\gamma(\E_1)=\cdots=\gamma(\E_n)=\gamma(\F)$ and $\G_i\in\Filt\{\E_i\}$ such that there are short exact sequences
\[0\to \F_{i+1}\to\F_i\to\G_i\to0\qquad\text{for }1\le i\le n\]
with $\Hom_X(\F_{i+1},\E_i)=0$, $\F_1=\F$ and $\F_{n+1}=0$. Since $\mu_{-K_X}(\E_1)=\cdots=\mu_{-K_X}(\E_n)$ holds and $\E_1,\cdots,\E_n$ are $\mu_{-K_X}$-semistable, we have
\[\Ext^2_X(\E_i,\E_j)\cong D\Hom_X(\E_j,\E_i(K_X))=0\]
for each $1\le i,j\le n$. Thus $\Ext_X^2(\G_i,\F_{i+1})=0$ holds and by induction, all $\F_i$ and $\G_i$ are rigid by Lemma \ref{heredrig}. Therefore by Lemma \ref{compfacrig}, each $\E_i$ is rigid. It is simple by $\gamma$-stability and
has $\Ext_X^2(\E_i,\E_i)=0$ by the preceding slope argument.
Corollary \ref{rigtorfbdl} now shows that it is an exceptional bundle. This proves the assertion.
\end{proof}

Now we can prove Theorem \ref{rigtorfss}.

\begin{proof}[Proof of Theorem \ref{rigtorfss}]
We prove the stronger statement that every rigid torsion-free sheaf
is the direct sum of its $\mu_{-K_X}$-Harder--Narasimhan factors.
Let $\F\in\G_n*\cdots*\G_1$ be this filtration, and put
$\mu_i:=\mu_{-K_X}(\G_i)$, so that $\mu_n>\cdots>\mu_1$.
Semistability and Serre duality give
\[
    \Hom_X(\G_i,\G_j)=0\quad(i>j),
    \qquad
    \Ext_X^2(\G_j,\G_i)=0\quad(i\geq j).
\]
Successive applications of Lemma \ref{heredrig} show that all
filtration terms and factors are rigid. In particular,
$\G_i\in\C_{\mu_i}$ by Lemma \ref{ssCmu}.
We argue by induction on $n$, the case $n=1$ being immediate.

Write $0\to\F'\to\F\to\G_1\to0$.
By induction, $\F'\cong\bigoplus_{i=2}^n\G_i$.
Set $\H:=\bigoplus_{i=3}^n\G_i\subseteq\F'$ (with $\H=0$
if $n=2$) and $\G:=\F/\H$. There is a short exact sequence
\[
    0\longrightarrow\G_2\xrightarrow{a}\G
      \xrightarrow{b}\G_1\longrightarrow0.
\]
The slopes of the factors of $\H$ are greater than those of
$\G_2$ and $\G_1$, so
$\Hom_X(\H,\G)=0$ and $\Ext_X^2(\G,\H)=0$.
Applying Lemma \ref{heredrig} to
$0\to\H\to\F\to\G\to0$ shows that $\G$ is rigid.

Let $c\colon\G_1\to\G_2[1]$ be the connecting morphism for
$0\to\G_2\to\G\to\G_1\to0$, and take
$f\colon\G_2\to\G_1[1]$. Since
\[
    fc[-1]\in\Hom_{\D^b(\Coh X)}(\G_1[-1],\G_1[1])
       =\Ext_X^2(\G_1,\G_1)=0,
\]
the long exact Hom sequence gives $g\colon\G\to\G_1[1]$
with $ga=f$. Moreover, $\Ext_X^2(\G,\G_2)=0$ by the
vanishings for $\G_1$ and $\G_2$ above. Thus $g=b[1]h$ for
some $h\colon\G\to\G[1]$. Rigidity gives $h=0$, hence $f=0$.
We have proved $\Ext_X^1(\G_2,\G_1)=0$, and Lemma \ref{ext1}
then gives $\Ext_X^1(\G_1,\G_2)=0$.

The extension class of $\F$ in
$\Ext_X^1(\G_1,\G_2\oplus\H)$ therefore has zero
$\G_2$-component. Consequently
$\F\cong\G_2\oplus\F''$, where
$0\to\H\to\F''\to\G_1\to0$ is exact.
The direct summand $\F''$ is rigid and has the remaining
Harder--Narasimhan factors. Induction proves
$\F\cong\bigoplus_{i=1}^n\G_i$.
If $\F$ is indecomposable, then $n=1$ and Lemma \ref{ssCmu}
gives $\F\in\C_{\mu_1}$. Its semistability follows from
Corollary \ref{extsemist}, and its $\Ext^2$ vanishes by
Serre duality and $K_X^2>0$. Thus it is also pretilting.
\end{proof}

As an application, we obtain a criterion for a tilting bundle to be
$2$-tilting in terms of its slopes. Part~(1) abstracts the final
vanishing argument in the proof of Theorem \ref{suff}.
Part~(2) gives a necessary condition with a non-strict inequality.

\begin{Thm}\label{2tiltslope}
Let $\E\in\vect X$ be a tilting bundle.
\begin{enumerate}
\item Assume that for all indecomposable direct summands $\E_1$ and $\E_2$ of $\E$, we have $\mu_{-K_X}(\E_1)<\mu_{-K_X}(\E_2(-K_X))$. Then $\E$ is $2$-tilting.
\item Assume $\E$ is $2$-tilting. Then for all indecomposable direct summands $\E_1$ and $\E_2$ of $\E$, we have $\mu_{-K_X}(\E_1)\leq\mu_{-K_X}(\E_2(-K_X))$.
\end{enumerate}
\end{Thm}
\begin{proof}
(1) Take indecomposable direct summands $\E_1$ and $\E_2$ of $\E$.
By Theorem \ref{rigtorfss}, they belong to the respective
categories $\C_{\mu_{-K_X}(\E_i)}$, as do their line-bundle twists
with the appropriately shifted slope. Since $\E$ is tilting, we have
\[\Ext^1_X(\E_2(-K_X),\E_1)\cong D\Ext^1_X(\E_1,\E_2)=0.\]
Thus by Lemma \ref{ext1}, we obtain
\[\Ext^1_X(\E_1,\E_2(-K_X))=0.\]
Since $\mu_{-K_X}(\E_2(-K_X))>\mu_{-K_X}(\E_1)>\mu_{-K_X}(\E_1(K_X))$,
semistability from Corollary \ref{extsemist} gives
\[\Ext^2_X(\E_1,\E_2(-K_X))\cong D\Hom_X(\E_2(-K_X),\E_1(K_X))=0.\]
Proposition \ref{dtiltsheaf} now implies that $\E$ is $2$-tilting.

(2) Since $\E$ is $2$-tilting, we have
\[\chi(\sheafHom_X(\E_1,\E_2(-K_X)))=\dim_k(\Hom_X(\E_1,\E_2(-K_X)))\text{ and}\]
\[\chi(\sheafHom_X(\E_1,\E_2))=\dim_k(\Hom_X(\E_1,\E_2)).\]
Suppose that $\mu_{-K_X}(\E_1)>\mu_{-K_X}(\E_2(-K_X))$ holds. Since $\E_1,\E_2(-K_X)$ and $\E_2$ are $\mu_{-K_X}$-semistable by Theorem \ref{rigtorfss} and $\mu_{-K_X}(\E_1)>\mu_{-K_X}(\E_2(-K_X))>\mu_{-K_X}(\E_2)$ holds, we have
\[\Hom_X(\E_1,\E_2(-K_X))=0\text{ and }\Hom_X(\E_1,\E_2)=0.\]
For vector bundles $U,V$ on a smooth projective surface,
the antisymmetric form of surface Riemann--Roch
\cite[Theorem~2.1.1, p.~2961]{Kul} is
\[
    \chi(U,V)-\chi(V,U)
    =\bigl(\rk(U)c_1(V)-\rk(V)c_1(U)\bigr)\cdot(-K_X),
\]
where $\chi(U,V):=\chi(U^\vee\otimes V)$.
Apply it to $U=\E_1$ and $V=\E_2(-K_X)$.
Serre duality also gives
$\chi(\E_2(-K_X),\E_1)=\chi(\E_1,\E_2)$. Hence
\begin{align*}
\chi(\sheafHom_X(\E_1,\E_2(-K_X)))-\chi(\sheafHom_X(\E_1,\E_2))&=\chi(\sheafHom_X(\E_1,\E_2(-K_X)))-\chi(\sheafHom_X(\E_2(-K_X),\E_1))\\
&=\rk(\E_2(-K_X))\rk(\E_1)(\mu_{-K_X}(\E_2(-K_X))-\mu_{-K_X}(\E_1))\\
&<0,
\end{align*}
which leads to a contradiction. Thus we obtain $\mu_{-K_X}(\E_1)\leq\mu_{-K_X}(\E_2(-K_X))$.
\end{proof}

\end{appendix}

\bibliographystyle{alpha} 
\begingroup
\raggedbottom
\renewcommand{\baselinestretch}{1}
\bibliography{reference}
\endgroup

\end{document}